\documentclass[11pt]{amsart}
\usepackage{amsfonts,amsmath,amsthm,amssymb,amscd,latexsym,enumerate,etoolbox,mathrsfs,comment, graphicx}
\usepackage[all]{xy}
\usepackage{thmtools, thm-restate}
\usepackage[normalem]{ulem}
\usepackage[T1]{fontenc}

\newcommand{\N}{\mathbb{N}}

\newcommand{\Z}{{\mathbb Z}}

\renewcommand{\phi}{\varphi}

\newcommand{\into}{\hookrightarrow}
\newcommand{\coker}{\textup{coker}}

\theoremstyle{plain}
   \newtheorem*{theorem*}{Theorem}
    \newtheorem{theorem}{Theorem}[section]
    \newtheorem{lemma}[theorem]{Lemma}
    \newtheorem{corollary}[theorem]{Corollary}
    \newtheorem{proposition}[theorem]{Proposition}
    
\theoremstyle{definition}

    \newtheorem{remark}[theorem]{Remark}

\theoremstyle{remark}

\DeclareMathOperator{\id}{id}

\DeclareMathOperator{\tr}{tr}

\DeclareMathOperator{\Ell}{Ell}

\title[$\mathrm{C}^*$-algebras from $\mathbb{Z}$-actions and orbit-breaking subalgebras]{Classifiable $\mathrm{C}^*$-algebras from minimal $\mathbb{Z}$-actions and their orbit-breaking subalgebras}

\author[R.J. Deeley, I.F. Putnam, K.R. Strung]
{Robin J. Deeley \and
Ian F. Putnam \and
Karen R. Strung}
\address{Department of Mathematics,
University of Colorado Boulder
Campus Box 395,
Boulder, CO 80309-0395, USA }
\email{robin.deeley@gmail.com}

\address{Department of Mathematics and Statistics,
University of Victoria,
Victoria, B.C., Canada V8W 3R4} 
\email{ifputnam@uvic.ca}

\address{Institute of Mathematics, Czech Academy of Sciences, \v{Z}itn\'a 25, 115 67 Prague, Czech Republic}
\email{strung@math.cas.cz}

\date{\today}
\subjclass[2010]{37B05, 46L35, 46L85, 19K99}
\keywords{minimal homeomorphisms, equivalence relations, classification of nuclear \mbox{$\mathrm{C}^{*}$-algebras}}
\thanks{RJD is currently funded by NSF Grant DMS 2000057 and was previously funded by Simons Foundation Collaboration Grant for Mathematicians number 638449. KRS is currently funded by GA\v{C}R project 20-17488Y and \mbox{RVO: 67985840} and part of this work was carried out while funded by Sonata 9 NCN grant 2015/17/D/ST1/02529 and a Radboud Excellence Initiative Postdoctoral Fellowship. IFP is supported in part by an NSERC Discovery Grant.}

\begin{document}

\begin{abstract} In this paper we consider the question of what abelian groups can arise as the $K$-theory of $\mathrm{C}^*$-algebras arising from minimal dynamical systems.  We completely characterize the $K$-theory of the crossed product of a space $X$ with finitely generated $K$-theory by an action of the integers and show that crossed products by a minimal homeomorphisms exhaust the range of these possible $K$-theories. Moreover, we may arrange that the minimal systems involved are uniquely ergodic, so that their \mbox{$\mathrm{C}^*$-algebras} are classified by their Elliott invariants. We also investigate the $K$-theory and the Elliott invariants of orbit-breaking algebras.  We show that given arbitrary countable abelian groups $G_0$ and $G_1$ and any Choquet simplex $\Delta$ with finitely many extreme points, we can find a minimal orbit-breaking relation such that the associated \mbox{$\mathrm{C}^*$-algebra} has $K$-theory given by this pair of groups and tracial state space affinely homeomorphic to $\Delta$.  We also improve on the second author's previous results by using our orbit-breaking construction to $\mathrm{C}^*$-algebras of minimal amenable equivalence relations with real rank zero that allow torsion in both $K_0$ and $K_1$. These results have important applications to the Elliott classification program for $\mathrm{C}^*$-algebras. In particular, we make a step towards determining the range of the Elliott invariant of the $\mathrm{C}^*$-algebras associated to \'{e}tale equivalence relations.\end{abstract}

\maketitle

\section{Introduction} \label{Sect:Intro}

From its very beginnings, the study of operator algebras has had close ties to the study of dynamical systems. Already in the first papers of Murray and von Neumann we have the group-measure space construction, which constructs a crossed product von Neumann algebra from a group acting by probability measure-preserving transformations on a standard probability space \cite{MvN:Rings4}. This construction was in turn imported into the noncommutative topological setting of \mbox{$\mathrm{C}^*$-algebras} via topological groups acting on locally compact Hausdorff spaces, see for example \cite{EffrosHahn1967, Takesaki1967, Zeller-Meier1968, MR584266}.

The symbiotic relationship between these two areas has allowed for powerful applications from dynamics to operator algebras, and vice versa. For example, operator algebraic versions of the Rokhlin lemma were crucial for the type classification of von Neumann algebra factors, while von Neumann classification allows one to conclude strong results about the orbit equivalence classes of ergodic group actions. Similarly, on the $\mathrm{C}^*$-algebraic side, Elliott's classification of approximately finite (AF) algebras by so-called dimension groups inspired work by Krieger on the classification, up to eventual conjugacy, of shifts of finite types \cite{Kri:DimFun}. Bratteli diagrams were a crucial ingredient in the Giordano--Putnam--Skau classification of minimal actions on a Cantor set \cite{GioPutSkau:orbit}. More generally, such constructions were used to study topological equivalence relations called AF-equivalence relations \cite{MR2054051}. Noncommutative versions of the Rokhlin lemma continue to provide important ways to study the structure of $\mathrm{C}^*$-algebras and their group actions (for example, \cite{Kis:RohlinUHF, Izu:FreeRok1, HirWinZac:RokDim} and many others).

As eluded to above, the links between operator algebras and dynamics have been particularly strong with respect to classification. In the von Neumann setting, a single probability measure-preserving ergodic transformation on a standard Borel space gives rise to the unique hyperfinite $\mathrm{II}_1$ factor $\mathcal{R}$. More remarkably, Connes, Feldmann and Weiss showed that every countable Borel equivalence relation is orbit-equivalent to a single transformation, and thus the corresponding von Neumann algebraic constructions all once again yield $\mathcal{R}$ \cite{MR662736}. 

Classification for the analogous $\mathrm{C}^*$-algebras proved, unsurprisingly, more involved. Rather than a unique such algebra, we are faced with a whole class of nonisomorphic $\mathrm{C}^*$-algebras. Inspired by von Neumann classification, Elliott initiated the classification programme for nuclear $\mathrm{C}^*$-algebras with his classification of AF algebras \cite{Ell:AF}. He later conjectured that many more simple separable $\mathrm{C}^*$-algebras might be classified by an invariant---now called the Elliott invariant---consisting of $K$-theory and tracial data.

The recent spectacular achievement of the Elliott program is the classification of all simple, nuclear, separable, unital,  infinite-dimensional \mbox{$\mathrm{C}^*$-algebras} with finite nuclear dimension which satisfy the universal coefficient theorem (UCT). The nuclear dimension is a $\mathrm{C}^*$-algebraic version of topological covering dimension and  $\mathrm{C}^*$-algebras with finite nuclear dimension are much better behaved than those without this property. The UCT is a tool which, loosely speaking, allows one to transfer information between $KK$-theory and $K$-theory. While it is a crucial assumption, it holds for all known examples of nuclear $\mathrm{C}^*$-algebras, in particular for those which arise from minimal dynamical systems by \cite{Tu:Groupoids}. Let us call a $\mathrm{C}^*$-algebra \emph{classifiable} if it is simple, separable, unital, infinite-dimensional, has finite nuclear dimension, and satisfies the UCT. 

Every minimal dynamical action of the integers $(X, \varphi)$, where $X$ is an infinite compact Hausdorff space $X$ and the $\mathbb{Z}$-action is induced by a homeomorphism $\varphi$, gives rise to a minimal equivalence relation on $X$ where the equivalence classes are the $\varphi$-orbits.  While in ergodic theory the analogous notions turn out to be essentially the same \cite{MR662736}, in topological dynamics, minimal equivalence relations are more general than the class of orbit equivalence relations associated to minimal homeomorphisms.  In this paper, we are interested in equivalence relations with a strong topological dynamical flavour: the so-called \emph{orbit-breaking} relations. These also arise from a minimal dynamical system $(X, \varphi)$ with the additional input of a non-empty closed subset $Y \subset X$. If $Y$ meets every $\varphi$-orbit at most once, then we can break the equivalence class corresponding to any orbit that passes through $Y$ into two distinct equivalence classes corresponding to the forward and backward orbits from $Y$. Since $Y$ meets every $\varphi$-orbit at most once, these equivalence classes are still dense in $X$, and, taken together with the equivalence classes corresponding to orbits that do not pass through $Y$, we obtain a new minimal equivalence relation on $X$.  Note that orbit-breaking relations are specific to minimal integer actions. In the case of minimal actions by more general groups (even free abelian groups), it is not obvious how to break orbits in such a way that the resulting subequivalence relation will be an open subset. This is one reason why we focus here on integer actions.

There are many such orbit-breaking relations for which the associated \mbox{$\mathrm{C}^{*}$-algebras} cannot be isomorphic to the $\mathrm{C}^*$-algebra of any minimal dynamical system. Breaking at a single point in a Cantor minimal system, for example, results in an AF algebra, which has trivial $K_1$-group. The $\mathrm{C}^*$-algebra of any minimal dynamical system, on the other hand, always has non-trivial $K_1$-group. In particular, if we restrict ourselves to single transformations, we do not get an analogue to the von Neumann algebra relationship, however, it is still possible that minimal equivalence relations might exhaust, up to isomorphism, all stably finite classifiable $\mathrm{C}^*$-algebras. 

Constructions of Giol and Kerr show that there are examples of crossed products associated to minimal homeomorphisms of infinite compact metric spaces which have infinite nuclear dimension \cite{GioKerr:Subshifts}. Thus not all such crossed product $\mathrm{C}^*$-algebras---or their orbit-breaking subalgebras---will fall within the scope of the classification theorem. However, Elliott and Niu provided a sufficient condition for finite nuclear dimension when they showed that whenever $(X, \varphi)$ has mean dimension zero, the crossed product will be isomorphic to itself tensored with the Jiang--Su algebra, $\mathcal{Z}$ \cite{EllNiu:MeanDimZero}. The Jiang--Su algebra is a simple, separable, unital, nuclear, infinite-dimensional $\mathrm{C}^*$-algebra whose Elliott invariant is the same as the Elliott invariant of $\mathbb{C}$.   A $\mathrm{C}^*$-algebra $A$ is called $\mathcal{Z}$-stable if it absorbs the Jiang--Su algebra $\mathcal{Z}$ tensorially, $A \cong A \otimes \mathcal{Z}$. Finite nuclear dimension of a \mbox{$\mathrm{C}^*$-algebra} $A$ turns out to be equivalent to $\mathcal{Z}$-stability of $A$, see \cite{CETWW} (which was based on earlier work in \cite{MR3273581, MR3418247}). When combined with the theorem of Elliott and Niu, results of Archey, Buck and Phillips imply that mean dimension zero ensures that any orbit-breaking algebra is also $\mathcal{Z}$-stable and hence has finite nuclear dimension \cite{ArBkPh-Z}.

It follows from this that a classification theorem for the $\mathrm{C}^*$-algebras associated to minimal dynamical systems is close at hand. One fundamental remaining piece is determining the range of the Elliott invariant for such $\mathrm{C}^*$-algebras. This in turn is key to understanding the relationship between minimal dynamical systems, minimal equivalence relations, and stably finite classifiable $\mathrm{C}^*$-algebras, in analogy to the relationship between ergodic measure-preserving transformations, countable Borel equivalence relations and the hyperfinite $\mathrm{II}_1$ factor.  Is every classifiable $\mathrm{C}^*$-algebras isomorphic to the $\mathrm{C}^*$-algebra of a minimal equivalence relation?

A natural starting point to addressing this problem is to understand the possible range of $K$-theory for crossed products by minimal homeomorphisms and their orbit-breaking $\mathrm{C}^*$-algebras.

In \cite[Theorem 2.2]{DPS:MinHomKtheory}, which can be seen as the dynamical companion to this paper, the authors showed that given any pair $G_0$ and $G_1$ of finitely generated abelian groups, there exists a compact metric space $X$ admitting a minimal homeomorphism and satisfying $K^0(X) \cong \Z \oplus G_0$ and $K^1(X) \cong G_1$. In particular, there are no $K$-theoretical obstructions preventing a space from admitting a minimal homeomorphism in the finitely-generated case.  The proof of the statement uses a skew-product construction of Glasner and Weiss \cite{GlaWei:MinSkePro} as well as the authors' construction of minimal homeomorphisms on ``point-like" spaces---spaces whose $K$-theory and cohomology are the same as a point \cite{DPS:DynZ}. In the same paper, a variation on the construction allows us more control over the map induced by the homeomorphism on $K$-theory, and allows us to completely exhaust possible $K$-theory for crossed products arising from minimal homeomorphisms on spaces $X$ with finitely generated $K$-theory, see Theorem~\ref{thm:Kpossibiliities} and Theorem~\ref{thm:RealizingCPKtheory}: \\

\emph{Suppose that $X$ has finitely generated $K$-theory and $\varphi : X \to X$ is a (minimal) homeomorphism. Then there are $d \in \mathbb{N}$ and finite abelian groups $F_0$ and $F_1$ such that $K_j(C(X) \rtimes_\varphi \mathbb{Z}) \cong \mathbb{Z}^d \oplus F_j$, for j=0,1. Moreover, for any $d \in \mathbb{N}$ and finite abelian groups $F_0$ and $F_1$, there exists an infinite connected compact metric space $X$ admitting a minimal homeomorphism $\varphi :X \to X$ such that 
 $K_j(C(X) \rtimes_\varphi \mathbb{Z}) \cong \mathbb{Z}^d \oplus F_j$, for $j = 0,1$.}\\
 
We are able to arrange that $A := C(X) \rtimes_{\varphi} \mathbb{Z}$ above has a unique tracial state and that the order structure of $K_0(A)$ is given by \[ (K_0(A),  K_0(A)_+,  [1]) \cong ( \mathbb{Z}^d \oplus F_0, \,  \mathbb{Z}_{> 0} \oplus \mathbb{Z}^{d-1} \oplus F_0 \cup (0_{\mathbb{Z}^d},0_{F_0}),(1, 0_{\mathbb{Z}^{d-1}}, 0_{F_0}))\]
while the pairing map in the Elliott invariant  $r : T(A) \to SK_0(A)$ satisfies $r(\tau)((n_1, \dots, n_d, g)) = n_1$.
 
Next, we investigate the range of the Elliot invariant of $\mathrm{C}^*$-algebras obtained from orbit-breaking relations associated to a minimal dynamical system. Here again, the existence of minimal homeomorphisms of ``point-like" spaces given in \cite{DPS:DynZ} play a key role.  We can arrange that such a system is uniquely ergodic. After breaking the orbit at a single point, we are left with a minimal equivalence relation whose $\mathrm{C}^*$-algebra is $^*$-isomorphic to the Jiang--Su algebra. Building on this, we construct projectionless $\mathrm{C}^*$-algebras from orbit-breaking equivalence relations, see Corollary~\ref{cor:fewProjections}:     \\

\emph{Let $G_0$ and $G_1$ be arbitrary countable abelian groups and $\Delta$ a Choquet simplex with finitely many extreme points. There exists an orbit-breaking equivalence relation $\mathcal{R}$ of an infinite compact metric space such that 
\[K_0(\mathrm{C}^*(\mathcal{R}))  \cong \mathbb{Z} \oplus G_{0}, \qquad K_1(\mathrm{C}^*(\mathcal{R}))  \cong G_{1}, \qquad T(\mathrm{C}^*(\mathcal{R})) \cong \Delta.\]
Moreover, the positive cone of $K_0(\mathrm{C}^*(\mathcal{R}))$ is given by \[K_0(\mathrm{C}^*(\mathcal{R}))_+ \cong \{ (n, z) \mid n=0, z=0, \text{ or }n > 0 \}.\]
}
Finally, we also apply our orbit-breaking technique to minimal homeomorphisms on non-homogeneous spaces constructed by the authors in \cite{DPS:nonHom}, which are generalisations of systems constructed by Floyd \cite{floyd1949} and Gjerde and Johansen \cite{FloGjeJohSys}.  In this case, thanks to the existence of a factor map onto a Cantor minimal system, the associated $\mathrm{C}^*$-algebras will always have real rank zero. This allows for an extensive strengthening of the second author's constructions in \cite{MR3770169},  see Theorem~\ref{thm:KthRR0}. Note in particular that the group $T$ below can have torsion:\\

\emph{Let $T$ be a countable abelian group, $G_0$ a simple acyclic dimension group and $G_1$ a countable abelian group. Then there exists an amenable equivalence relation $\mathcal{R}$ such that $\mathrm{C}^*(\mathcal{R})$ has real rank zero and }
\[ (K_0(\mathrm{C}^{*}(\mathcal{R})), K_0(\mathrm{C}^*(\mathcal{R}))_+, [1]) \cong (T \oplus G_0, (G_0)_+, 1_{G_0}), \quad K_1(\mathrm{C}^*(\mathcal{R})) \cong G_1.\]
\vspace{.25mm}

Our results fit within the general question of determining which stably finite classifiable $\mathrm{C}^*$-algebras can be realized as the $\mathrm{C}^*$-algebra of a principal \'etale groupoid. Spectacular recent work on this question has been obtained by Li \cite{Li:Cartan} (also see \cite{AusMitra:GrpModGelfDual, BarLi:CartanUCT1, BarLi:CartanUCT2}). The main result in \cite{Li:Cartan} is that every stably finite classifiable $\mathrm{C}^*$-algebra can be realized as the $\mathrm{C}^*$-algebra of a \emph{twisted} principal \'etale groupoid.  There, rather than focusing on constructions coming from dynamics, Li mimics known inductive limit constructions cor the $\mathrm{C}^*$-algebras at the level of the groupoids.  Another difference is that the twist on the groupoid is only non-trivial when the $K_0$-group of the corresponding \mbox{$\mathrm{C}^*$-algebra} is torsion free. Here we are able to have torsion in $K$-theory without requiring any twists. 

The paper is organized as follows. In Section~\ref{Sect:Prelim}, we recall some facts about minimal dynamical systems, crossed products, \'etale groupoids, and the associated $\mathrm{C}^*$-algebras. In particular we recall the orbit-breaking equivalence relation that gives rise to orbit-breaking algebras. In Section~\ref{sec:classification}, we discuss the classification program for $\mathrm{C}^*$-algebras, necessary conditions for crossed products by minimal homeomorphisms and their orbit-breaking subalgebras to be classifiable, and recall the tools required for calculating their Elliott invariants. Section~\ref{Sect:Skew} focusses on computations for crossed products arising from the skew product systems constructed in \cite{DPS:MinHomKtheory}. We prove the existence of minimal homeomorphisms exhausting the possible $K$-theoretical range of crossed products of $C(X)$ when $X$ has finitely generated $K$-theory. In the final three sections, we turn our attention to minimal orbit-breaking relations and their $\mathrm{C}^*$-algebras. Technical results on the K-theory of orbit-breaking subalgebras are then proved in Section \ref{TechResultSection}.  Section~\ref{sec:orbitbreaking} deals with projectionless $\mathrm{C}^*$-algebras arising from skew-product systems and systems on point-like spaces. In Section~\ref{Subsect:RR0} we look at real rank zero $\mathrm{C}^*$-algebras obtained from the non-homogeneous minimal systems constructed in \cite{DPS:nonHom}. We expand on the results of the second author in \cite{MR3770169} to obtain examples with torsion in $K_0$ and arbitrary $K_1$. 

\section{Dynamical systems and \'etale groupoids} \label{Sect:Prelim}

By a \emph{dynamical system}, we mean a compact  Hausdorff space $X$, which for the purposes of this paper is always assumed to be metrizable, equipped 
with a homeomorphism $\varphi : X \to X$.  The \emph{$\varphi$-orbit} of a point $x$ in $X$ is the set $\{ \varphi^{n}(x) 
\mid n \in \mathbb{Z} \}$.

In the sequel, $\varphi$ will always induce a \emph{free} action of the integers on $X$, which is to say that if there exists $x \in X$ with $\varphi^{n}(x)=x$, then $n=0$. The homeomorphism $\varphi$ is \emph{minimal}, or $(X, \varphi)$ is a \emph{minimal dynamical system}, if the only closed subsets $Y \subseteq X$ with $\varphi(Y)=Y$ are $X$ and the empty set. This is equivalent to the condition that every  $\varphi$-orbit is dense in $X$. If $X$ is infinite, then any minimal dynamical system induces a free action.
 
Given two dynamical systems $(X, \varphi)$ and $(Y, \psi)$ a map $\pi : X \to Y$ is called  a \emph{factor map} if $\pi$  is a continuous surjection satisfying $\pi \circ \varphi = \psi \circ \pi$.   In this case, $(Y, \psi)$ is called a \emph{factor} of $(X, \varphi)$ and   $(X, \varphi)$ is a called an \emph{extension} of $(Y, \psi)$. If $\pi : (X, \varphi) \to (Y, \psi)$ is a factor map then  $\pi : (X, \varphi^n) \to (Y, \psi^n)$ is also a factor map for any integer $n$. A factor map $\pi : X \to Y$ is 
\emph{almost one-to-one} if it is one-to-one on a residual subset of $X$.

To minimal dynamical system $(X, \varphi)$ one may associate a topological groupoid, called the \emph{transformation groupoid} $X \rtimes_{\varphi} \mathbb{Z}$, see for example \cite{MR584266}. Here, since our systems will always be free,  it will be more convenient to reformulate the transformation groupoid as the orbit equivalence relation on $X$. Given a free dynamical system $(X, \varphi)$,  define the orbit equivalence relation 
\[
\mathcal{R}_{\varphi} := \{ (x, \varphi^{n}(x) ) \mid x \in X, n \in \mathbb{Z} \},
\]
which is an equivalence relation whose equivalence classes are simply the $\varphi$-orbits. As  the dynamical system is free, the map 
\[ X \rtimes_{\varphi} \mathbb{Z} \to \mathcal{R}_{\varphi}, \quad  (x, n) \mapsto (x, \varphi^{n}(x)) \] is 
 a bijection. We endow $X \rtimes_{\varphi} \mathbb{Z}$ with the product topology and equip $\mathcal{R}_{\varphi}$ with a topology via this map, 
 that is, $\mathcal{R}_{\varphi}$ is given the unique topology which makes this map a 
 homeomorphism. This endows the equivalence relation $\mathcal{R}_{\varphi}$ with an \'etale topology: the topology on an equivalence relation $\mathcal{R} \subset X \times X$ is \emph{\'etale} if the maps $\mathcal{R} \to X$ given by $(x,y) \mapsto x$ and $(x,y) \mapsto y$ are local homeomorphisms. (Note that in the topological groupoid literature the term \emph{equivalence relation} is sometimes reserved for an equivalence relation $\mathcal{R} \subset X \times X$ with topology inherited from the product topology on $X \times X$, whereas an equivalence relation equipped with any another topology is called a \emph{principal groupoid}.)
 
An equivalence relation on a compact metric space $X$ is \emph{minimal} if, for every $x \in X$, the equivalence class of $x$ is dense in $X$.  Observe that if $(X, \varphi)$ is a free dynamical system, then $\mathcal{R}_{\varphi}$ is minimal if and only if $(X, \varphi)$ is minimal.

Suppose $(X, \varphi)$ is a minimal dynamical system and $Y \subseteq X$ is a closed 
non-empty subset of $X$. We say that $Y$ \emph{meets every orbit at most once} if $\varphi^n(Y) \cap Y = \emptyset$ for each $n \neq 0$. Given a minimal dynamical system $(X, \varphi)$ and closed non-empty subset $Y \subset X$ meeting every orbit at most once, we construct another equivalence relation using the groupoid construction of \cite[ Example 2.6]{Put:K-theoryGroupoids}, 
which was originally used in the $\mathrm{C}^*$-algebraic setting in \cite{Putnam:MinHomCantor}. 
 Define $\mathcal{R}_Y \subseteq \mathcal{R}_{\varphi}$ to be the subequivalence relation obtained from splitting every orbit that passes through $Y$ into two equivalence classes. Specifically,
 \[ 
 \mathcal{R}_{Y} = \mathcal{R}_{\varphi} \setminus \{ (\varphi^{k}(y), \varphi^{l}(y)) 
 \mid y \in Y, \quad  l < 1 \leq k \text{ or }  k < 1 \leq l \}.
 \]
 
 It is a simple matter to check that this is an open  subequivalence relation of $\mathcal{R}_{\varphi}$ and  hence is also \'{e}tale. If some  $\varphi$-orbit does not meet $Y$, then  it is an equivalence class in both $\mathcal{R}_{\varphi}$ and $\mathcal{R}_{Y}$. If the  orbit does meet $Y$, say at the point $y$, then its $\varphi$-orbit becomes two distinct equivalence classes in $\mathcal{R}_{Y}$, namely $\{ \varphi^{n}(y) \mid n \geq 1 \}$ and   $\{ \varphi^{n}(y) \mid n \leq 0 \}$. In this sense, the orbit is ``broken'' in two at the point $y$.

Since $X$ is compact, for any point $x \in X$ both its forward orbit and backward orbit are dense in $X$. Thus we make the following observation.

\begin{proposition} \label{equDense}
If $(X, \varphi)$ is a minimal system, $X$ is infinite, and $Y$ is a closed non-empty subset of $X$ that meets each orbit at most once, then
$\mathcal{R}_Y$ is minimal. 
\end{proposition}

Let $\mathcal{R} \subset X \times X$ be an equivalence relation equipped with an \'etale topology, such as $\mathcal{R}_{\varphi}$ or $\mathcal{R}_{Y}$.  Using the method of Renault in  \cite[Chapter II]{MR584266} we  construct the reduced groupoid $\mathrm{C}^*$-algebra $\mathrm{C}^*_r(\mathcal{R})$ as follows. Equip the linear space $C_c(\mathcal{R})$ of compactly supported continuous functions $\mathcal{R} \to \mathbb{C}$ with a product and involution given by
\begin{eqnarray}
 (fg)(x, x') &=& \sum_{\substack{y \in X\\(x,y), (y,x') \in \mathcal{R}}} f(x, y)g(y,x'), \\
(f^*)(x,x') &=& \overline{f(x',x)},
\end{eqnarray}
for $f, g \in C_c(\mathcal{R})$ and $(x,x') \in \mathcal{R}$. This makes $C_c(\mathcal{R})$ into a $^*$-algebra. Let $\ell^2(\mathcal{R})$ denote the Hilbert space of square summable functions on $\mathcal{R}$. Then we define the regular representation $\lambda : C_c(\mathcal{R}) \to \mathcal{B}(\ell^2(\mathcal{R}))$ by
\[ (\lambda(f) \xi)(x, x') =  \sum_{\substack{y \in X\\(x,y), (y,x') \in \mathcal{R}}} f(x, y)\xi(y,x'),\]
for $f \in  C_c(\mathcal{R})$, $\xi \in \ell^2(\mathcal{R})$ and $(x,x') \in \mathcal{R}$. The reduced groupoid $\mathrm{C}^*$-algebra $\mathrm{C}^*_{r}(\mathcal{R})$ is then the closure of $\lambda(C_c(\mathcal{R}))$ with respect to the norm on $\mathcal{B}(\ell^2(\mathcal{R}))$. There is also a full groupoid $\mathrm{C}^*$-algebra, however if the  groupoid is amenable its  full and reduced $\mathrm{C}^{*}$-algebras coincide. In this case, we will simply denote the groupoid $\mathrm{C}^*$-algebra as $\mathrm{C}^*(\mathcal{R})$.  For any dynamical system $(X, \varphi)$, the equivalence relation $\mathcal{R}_{\varphi}$ is always amenable \cite[Example II.3.10]{MR584266}. It follows that if $Y$ is a closed non-empty subset meeting every $\varphi$-orbit at most once, then $\mathcal{R}_Y$ is also amenable, since it is an open subequivalence relation of $\mathcal{R}_{\varphi}$ \cite[Proposition 5.1.1]{MR1799683}.

Let $(X, \varphi)$ be a minimal dynamical system and a non-empty closed subset $Y$ meeting every $\varphi$-orbit at most once. The $\mathrm{C}^*$-algebra associated with $\mathcal{R}_{\varphi}$ is isomorphic to the crossed product $\mathrm{C}^*$-algebra $C(X) \rtimes_{\varphi} \mathbb{Z}$. Since $\mathcal{R}_{Y}$ is an open subequivalence relation of $\mathcal{R}_{\varphi}$, there is an inclusion  $C_c(\mathcal{R}_{Y} ) \subseteq  C_c(\mathcal{R}_{\varphi})$  by simply setting a function $f \in C_c(\mathcal{R}_Y)$ to zero on $\mathcal{R}_{\varphi} \setminus \mathcal{R}_Y$. This extends to a  unital inclusion of $\mathrm{C}^*$-algebras. If $A := \mathrm{C}^*(\mathcal{R}_{\varphi}) \cong C(X) \rtimes_{\varphi} \mathbb{Z}$, we will usually denote $\mathrm{C}^*(\mathcal{R}_Y)$ by $A_Y$.  Observe that both $A$ and $A_Y$ contain copies of $C(X)$, so, taken together with the inclusion $\iota : A_Y \into A$,  we have a commutative diagram
\begin{displaymath} 
\xymatrix{ & C(X) \ar[dl]_{i_1} \ar[dr]^{i_2} &\\
A_{Y} \ar[rr]_{\iota} & & A.}
\end{displaymath}
 
Another description of the $\mathrm{C}^*$-subalgebra $A_{Y}$ often encountered in the literature  is given by the $\mathrm{C}^*$-subalgebra of $A$ generated by subsets $C(X)$ and $u C_0(X \setminus Y)$,
\[
A_Y = C^*(C(X), uC_0(X \setminus Y)) \subseteq C(X) \rtimes_{\varphi} \mathbb{Z}, 
\]
where $C(X) \subset A$ is the standard inclusion
and $u$ is the unitary inducing $\varphi$, that is, the unitary $u \in A$ satisfying $ ufu^* = f \circ \varphi^{-1}$ whenever $ f\in C(X)$.

\section{The classification program} \label{sec:classification}

Let $A$ be a simple, separable, unital, nuclear $\mathrm{C}^*$-algebra. The \emph{Elliott invariant} of $A$, denoted $\Ell(A)$, is the 6-tuple
\[ \Ell(A) := (K_0(A), K_0(A)_+, [1_A], K_1(A), T(A), r_A : T(A) \to S(K_0(A))),\]
where $r_A : T(A) \to S(K_0(A))$ maps a tracial state $\tau$ to the state on the ordered abelian group $(K_0(A), K_0(A)_+,  [1_A])$, defined by $(\tau)([p]-[q]) =  \tau(p) - \tau(q)$, for projections $p, q \in M_{\infty}(A)$. (Here, by abuse of notation, $\tau$ denotes $\tau \otimes \tr_{M_n}$ for appropriate $n \in \mathbb{N}$).

The classification theorem stated below is the culmination of many years of work.

\begin{theorem}[see \cite{CETWW, BBSTWW:2Col, EllGonLinNiu:ClaFinDecRan, GongLinNiue:ZClass,  TWW}]
 \label{ClassThm} Let $A$ and $B$ be separable, unital, simple, infinite-dimensional \mbox{$\mathrm{C}^*$-algebras} with finite nuclear dimension and which satisfy the UCT. Suppose there is an isomorphism 
\[ \gamma : \Ell(A) \to \Ell(B).\]
Then there is a $^*$-isomorphism $\Gamma : A \to B$, which is unique up to approximate unitary equivalence and satisfies $\Ell(\Gamma) = \gamma$.
\end{theorem}

 For a precise definition of the Universal Coefficient Theorem (UCT), see \cite{RosSho:UCT}. It is satisfied for all known nuclear $\mathrm{C}^*$-algebras, in particular for any of the $\mathrm{C}^*$-algebras in the sequel. Nuclear dimension is a type of topological dimension for $\mathrm{C}^*$-algebras which, in the case of a commutative $\mathrm{C}^*$-algebra, coincides with the covering dimension of its spectrum. Nuclear dimension was introduced by Winter and Zacharias in \cite{WinterZac:dimnuc} as a generalization of previous work in \cite{KirWinter:dr} and \cite{Win:cpr}.  Restricting to finite nuclear dimension in the classification theorem given below is thus akin to restricting to finite covering dimension. This avoids any pathologies that can occur. In particular, any simple infinite-dimensional $\mathrm{C}^*$-algebra $A$ with weakly unperforated $K_0$-group will satisfy $\Ell(A) \cong \Ell(A \otimes \mathcal{Z})$, where $\mathcal{Z}$ is the \emph{Jiang--Su algebra} \cite[Theorem 1]{GongJiangSu:Z}. The Jiang--Su algebra was constructed in \cite{JiaSu:Z}. It is a simple, unital, nuclear, infinite-dimensional $\mathrm{C}^*$-algebra whose Elliott invariant is the same as the Elliott invariant of $\mathbb{C}$. We say a $\mathrm{C}^*$-algebra $A$ is \emph{$\mathcal{Z}$-stable} if $A \cong A \otimes \mathcal{Z}$. It is now known that finite nuclear dimension is equivalent to $\mathcal{Z}$-stability for separable, unital, nuclear, simple and infinite-dimensional \mbox{$\mathrm{C}^*$-algebras} \cite{Win:Z-stabNucDim, CETWW, BBSTWW:2Col, MR3273581, MR3418247}.

We  denote the class of all  unital, simple, infinite-dimensional \mbox{$\mathrm{C}^*$-algebras} with finite nuclear dimension and which satisfy the UCT by $\mathcal{C}$, that is,
 \[ \mathcal{C} := \{ A \mid A \text{ a classifiable $\mathrm{C}^*$-algebra}\}.\]
 
 \begin{theorem} \label{classifiableTheorem}
Let $(X, \varphi)$ be a minimal dynamical system with $X$ an infinite compact metric space. Suppose that $A := C(X) \rtimes_{\varphi} \mathbb{Z} \in \mathcal{C}$. If $Y \subset X$ is a non-empty closed subset meeting every orbit at most once, then $A_Y \in \mathcal{C}$. In particular, this is the case if  $(X, \varphi)$ has at most countably many ergodic measures or $\dim(X) < \infty$. 
\end{theorem}

\begin{proof} Assume that $C(X) \rtimes_{\varphi} \mathbb{Z}$ is classifiable. That $A_Y$ is separable and unital is clear. Since $\mathcal{R}_Y$ is amenable, $A_Y$ is nuclear (see for example \cite[Corollary 6.2.14]{MR1799683}) and satisfies the UCT \cite{Tu:Groupoids}. Moreover, since $Y$ meets every $\varphi$-orbit at most once, Proposition \ref{equDense} implies that $A_Y$ is a simple $\mathrm{C}^*$-algebra. Finally, since $A_Y$ is a centrally large subalgebra of the nuclear $\mathrm{C}^*$-algebra $A$ (see \cite[Section 4]{ArcPhi2016}) it follows that $A_Y$ is also $\mathcal{Z}$-stable \cite[Theorem 2.4]{ArBkPh-Z}. That $C(X) \rtimes_{\varphi} \mathbb{Z}$ is when $(X, \varphi)$ is uniquely ergodic is the main result of Toms and Winter in \cite{TomsWinter:minhom, TomsWinter:PNAS}, while classification in the case for $\dim(X) < \infty$ is due to Lin in \cite{Lin:MinDyn}. Both results also follow from the results of Elliott and Niu in \cite{EllNiu:MeanDimZero}, as both these conditions imply $(X, \varphi)$ has mean dimension zero (see  \cite{MR1742309,  LinWeiss:MTD} for more on mean dimension).
\end{proof}
 
Determining the Elliott invariant for a given $\mathrm{C}^*$-algebra requires computing its $K$-theory. For crossed products by $\mathbb{Z}$, the principal tool for the computing $K$-theory is the Pimsner--Voiculescu exact sequence \cite{Pimsner1980}, which, for $C(X) \rtimes_{\varphi} \mathbb{Z}$, is given by
\begin{displaymath} 
\xymatrix{
K^0(X) \ar[r]^-{\id-\varphi^*} & K^0(X) \ar[r] &  K_0(C(X) \rtimes_{\varphi} \Z)  \ar[d]^{\partial_{\text{PV}}}\\
 K_1(C(X) \rtimes_{\varphi} \Z) \ar[u]^{\partial_{\text{PV}}} &  K^1(X) \ar[l] & K^1(X) \ar[l]_-{\id-\varphi^*}.}
\end{displaymath}

For orbit-breaking subalgebras $\mathrm{C}^*(\mathcal{R}_Y)$, we use the following exact sequence which relates the $K$-theory of $\mathrm{C}^*(\mathcal{R}_Y)$ to the $K$-theory of  $C(Y)$ and $C(X) \times_{\varphi} \mathbb{Z}$.

\begin{theorem}[Theorem 2.4 and Example 2.6 of \cite{Put:K-theoryGroupoids}] \label{PutExtSeq} 
Let $(X, \varphi)$ be a minimal dynamical system and $Y \subseteq X$ a closed non-empty subset of $X$ which meets 
every orbit at most once. Let $ A:= C(X) \rtimes_{\varphi} \mathbb{Z}$ and let $\iota : A_Y \into A$ be the inclusion map. Then there exists a six-term exact sequence
\begin{displaymath} 
\xymatrix{ K^0(Y) \ar[r] & K_0(A_Y) \ar[r]^-{\iota_*} & K_0(C(X) \rtimes_{\varphi} \mathbb{Z}) \ar[d]^{\partial_{\text{OB}}}\\
K_1(C(X) \rtimes_{\varphi} \mathbb{Z} ) \ar[u]^{\partial_{\text{OB}}} & K_1(A_Y) \ar[l]_-{\iota_*} & K^1(Y). \ar[l]}
\end{displaymath}
\end{theorem}

The details for the remaining maps in this six-term exact sequence are given in \cite{Put:K-theoryGroupoids} and also discussed further in Section~\ref{TechResultSection}.

In addition to the six-term exact sequence which allows us to relate the \mbox{$K$-theory} of an orbit-breaking subalgebra $A_Y$ to the containing crossed product $A = C(X) \rtimes_{\varphi} \mathbb{Z}$, we can also compare their tracial state spaces. In fact, by \cite[Theorem 6.2, Theorem 7.10]{Phi:LargeSubalgebras}, the restriction map 
\[ T(A) \to T(A_Y), \quad \tau \mapsto \tau|_{A_Y} \]
is bijective.

\section{Projectionless crossed products by minimal homeomorphisms} \label{Sect:Skew}

In this section, we consider projectionless crossed products arising from the minimal skew product systems constructed in \cite{DPS:MinHomKtheory}. These constructions are based on work of Glasner and Weiss in \cite{GlaWei:MinSkePro} and Fathi and Herman in \cite{FatHer:Diffeo}. 

To begin, let us consider the possible abstract $K$-theory groups of crossed products $C(X) \rtimes_{\varphi} \mathbb{Z}$ obtained from a (minimal) homeomorphism $\varphi$ when the underlying space $X$ has finitely generated topological $K$-theory. 

\begin{theorem} \label{thm:Kpossibiliities}
Suppose that $X$ is a compact metric space with finitely generated $K$-theory and $\phi$ is a homeomorphism of $X$. Then there exists $d\ge 1$ and finite abelian groups $F_0$ and $F_1$ such that 
\[
K_0(C(X) \rtimes_{\phi} \Z) \cong \Z^d \oplus F_0 \hbox{ and }K_1(C(X) \rtimes_{\phi} \Z)\cong \Z^d\oplus F_1.
\]
\end{theorem}
\begin{proof}
The Pimsner--Voiculescu exact sequence and the fact that $K^j(X)$ is finitely generated imply that $K_j(C(X) \rtimes_{\psi} \Z)$ is also finitely generated, $j=0,1$. That the free part of the degree zero group is the same as the free part of the degree one group follows from the rank--nullity theorem. Finally, $d\ge 1$ because the class of the unitary giving the actions generates a copy of $\Z$ in $K_1(C(X) \rtimes_{\psi} \Z)$.
\end{proof}

We will show that for every such pair of groups given by Theorem~\ref{thm:Kpossibiliities}, there exists a minimal dynamical system $(X, \varphi)$ such that the $K$-theory $C(X) \rtimes_{\varphi} \mathbb{Z}$ is precisely these two groups. To do so, we use results from \cite{DPS:MinHomKtheory}. Here $\tilde{K}^0(X)$ denotes the reduced $K_0$-group of $X$.

\begin{theorem}\cite[Theorem 5.9]{DPS:MinHomKtheory} \label{Thm:MainDynamic} Let $G_0$ and $G_1$ be finitely generated abelian groups and  $\sigma_0: G_0 \rightarrow G_0$, $\sigma_1: G_1 \rightarrow G_1$ finite order automorphisms.  Suppose there exists a pointed connected finite CW-complex $W$ such that 
\[\tilde{K}^0(W)\cong G_0, \text{ and } K^1(W)\cong G_1,\]
and basepoint-preserving finite order homeomorphism 
\[ \beta_W: W \rightarrow W \text{ such that } \beta_W^*=\sigma_*. \]
Then there exists a a compact metric space $X$ such that $K^0(X)\cong \Z \oplus G_0$, $K^1(X) \cong G_1$, and a minimal homeomorphism $\varphi : X \to X$ such that  $\varphi^* = \sigma_*$. Moreover, we can arrange that  $H^1(X) = \{0\}$ and that $\varphi$ is uniquely ergodic. \end{theorem}

 \begin{lemma} \label{cyclic}
Given $n\in \N$ there exists a connected finite CW-complex $W$ with vanishing $K^1$-group and $\beta: W \rightarrow W$ a finite order homeomorphism such that the induced map on $\tilde{K}^0(W)$ satisfies
\[
\ker(\id-\beta^*)\cong \{ 0 \} \hbox{ and }\coker(\id-\beta^*)\cong \Z/n\Z.
\]
\end{lemma}
\begin{proof}
Let $W_0 = S^1 \vee S^1 \vee \cdots \vee S^1$ be the wedge of $n-1$ copies of $S^1$ with basepoint the wedge point, and denote by $(x, i) \in S^1 \times \{1, \dots, n\}$ the point $x$ in the $i^{\text{th}}$ copy of $S^1$. Define $f : W_0 \to S^1$ to be the map $(x, i) \mapsto x$. The connected CW-complex $W$ is given by $C_f$, the reduced mapping cone of the map $f$. Observe that $C_f$ is a pointed space such that $\tilde{K}^0(W) \cong \mathbb{Z}^{n-1}$ and $\tilde{K}^1(W) = 0$ and there exists a basepoint preserving homeomorphism $\beta : W \to W$ which induces the map on $K$-theory given by multiplication of an element in $\mathbb{Z}^{n-1}$ by the matrix 
\[
B=\left[ \begin{array}{cccccc}0 & 0 & 0 & \cdots & 0 & -1 \\ 1 & 0 & 0 & \cdots & 0 & -1 \\ 0 & 1 & 0 & \cdots & 0 & -1 \\ 0 & 0 & 1 & \cdots & 0 & -1 \\ \vdots & \vdots & \vdots & \ddots & \vdots & \vdots \\ 0 & 0 & 0 & \cdots & 1 & -1 \end{array}\right].
\]
For further details, see  \cite[Example 5.6, Lemma 5.7]{DPS:MinHomKtheory}. It follows that the map $\id-\beta^*$ on $\tilde{K}^0(W)$ is given by
\[
\left[ \begin{array}{cccccc}1 & 0 & 0 & \cdots & 0 & 1 \\ -1 & 1 & 0 & \cdots & 0 & 1 \\ 0 & -1 & 1 & \cdots & 0 & 1 \\ \vdots & \vdots & \vdots & \ddots & \vdots & \vdots   \\ 0 & 0 & 0 & \cdots & 1 & 1 \\ 0 & 0 & 0 & \cdots & -1 & 2 \end{array}\right],
\]
which has Smith normal form given by
\[
\left[ \begin{array}{cccccc}1 & 0 & 0 & \cdots & 0 & 0 \\ 0 & 1 & 0 & \cdots & 0 & 0 \\ 0 & 0 & 1 & \cdots & 0 & 0 \\ \vdots & \vdots & \vdots & \ddots & \vdots & \vdots   \\ 0 & 0 & 0 & \cdots & 1 & 0 \\ 0 & 0 & 0 & \cdots & 0 & n \end{array}\right].
\]
Thus $\ker(\id - \beta^*) \cong \ker(I - B) = \{0\}$ and $\coker(\id - \beta^*) \cong \coker(I - B) \cong \mathbb{Z}/n\mathbb{Z}$, so $W$ and $\beta$ satisfy the requirements of the lemma.
\end{proof}

\begin{lemma}
Suppose $d\in \N \setminus \{0 \}$, and $F_0$, $F_1$ are finite abelian groups. Then there exist a connected finite CW-complex, $W$ and $\beta: W\rightarrow W$ a finite order homeomorphism such that
the induced map on $K^0(W)$ satisfies
\[
\ker(\id-\beta^*)\cong \Z^d \hbox{ and }\coker(\id-\beta^*)\cong  \Z^d\oplus F_0,
\]
and the induced map on $K^1(W)$ satisfies
\[
\ker(\id-\beta^*)\cong \{ 0 \} \hbox{ and }\coker(\id-\beta^*)\cong F_1.
\]
\end{lemma}
\begin{proof}
 Let $V$ to be the wedge of $(d-1)$-spheres with the homeomorphism, $\beta_0=\id$. Then $K^0(V) \cong \Z^d$, $K^1(V)\cong \{ 0 \}$ and $\id-\beta_0^*$ is the zero map on both groups. Hence the induced map on $K^0(W_0)$ satisfies
\[
\ker(\id-\beta_0^*)\cong \Z^d \hbox{ and }\coker(\id-\beta_0^*)\cong  \Z^d,
\]
and the induced map on $K^1(W_0)$ satisfies
\[
\ker(\id-\beta_0^*)\cong \{ 0 \} \hbox{ and }\coker(\id-\beta_0^*)\cong \{ 0 \}.
\]
Since $F_0$ is finite, there exist $m \in \mathbb{N}$ and cyclic groups $C_1, \dots, C_m$ such that $F_0 = \bigoplus_{j=1}^m C_j$. For each $j=1, \dots, m$, apply Lemma~\ref{cyclic} to get a pointed, connected, finite CW-complex $X_j$ with $K_1(X_j) = 0$, and a finite order homeomorphism $\alpha_j : X_j \to X_j$ such that the induced map on $\tilde{K}^0(X_j)$ satisfies
\[ \ker(\id - \alpha_j^*) = \{0\} \text{ and } \coker(\id - \alpha_j^*) \cong C_j.\]

By \cite[Theorem 5.5]{DPS:MinHomKtheory}, there exist a pointed, connected, finite CW-complex $X$ with $K_1(X) = \{0\}$ and $\beta_1 : X \to X$ a finite order basepoint-preserving homeomorphism with $\beta^*_1 = \alpha^*_1 \oplus \cdots \oplus \alpha^*_m$. In particular,
\[ \ker(\id - \beta^*_1) = \{0\} \text{ and } \coker(\id - \beta_1^*) \cong F_0.\]

Repeating the above for the finite abelian group $F_1$, we obtain a pointed, connected, finite CW-complex $\tilde{Y}$ with $K^1(\tilde{Y}) = \{0\}$ and a finite order basepoint-preserving homeomorphism $\tilde{\beta}_2 : \tilde{Y} \to \tilde{Y}$ such that $\ker(\id - \tilde{\beta}^*_1) = \{0\}$ and $\coker(\id - \tilde{\beta}_1^*) \cong F_0$.  Let $Y := S\tilde{Y}$, the suspension of $\tilde{Y}$.  Let $\beta_2  : Y \to Y$ be the finite order homeomorphism induced by $\beta$. Then $K^0(Y) = 0$ and the map induced $\beta_2$ on $K^1(Y)$ satisfies
\[ \ker(\id - \beta^*_2) = \{0\} \text{ and } \coker(\id - \beta_2^*) \cong F_1.\]

Let $W := V \vee X \vee Y$ and $\beta := \beta_0 \vee \beta_1 \vee \beta_2$. Using a similar argument to that of \cite[Theorem 5.5]{DPS:MinHomKtheory}, it now follows that $W$ and $\beta$ satisfy the requirements of the lemma. 
\end{proof}

\begin{theorem} \label{thm:RealizingCPKtheory}
Suppose $d\in \N \setminus \{0\}$ and $F_0$, $F_1$ are finite abelian groups. Then there exists an infinite connected compact metric space $X$ with finitely generated $K$-theory and $\tilde{\beta}: X\rightarrow X$ a minimal homeomorphism such that
\[
K_0(C(X)\rtimes_{\tilde{\beta}} \Z) \cong \Z^d \oplus F_0 \hbox{ and }K_1(C(X)\rtimes_{\tilde{\beta}}\Z)\cong \Z^d\oplus F_1.
\]
\end{theorem}
\begin{proof}
By the previous lemma, there exists a connected finite CW-complex $W$ and a finite order homeomorphism $\beta:W \rightarrow W$ such that the induced map on $K^0(W)$ satisfies
\[
\ker(\id-\beta^*)\cong \Z^d \hbox{ and }\coker(\id-\beta^*)\cong \Z^d\oplus F_0
\]
while the induced map on $K^1(W)$ satisfies
\[
\ker(\id-\beta^*)\cong \{ 0 \} \hbox{ and }\coker(\id-\beta^*)\cong F_1.
\]
Now apply Theorem~\ref{Thm:MainDynamic} to the finite order homeomorphism $\beta:  W  \rightarrow W $ to obtain an infinite compact metric space $X$ with $K^j(X) \cong K^j(W)$, $j = 0,1$, and a minimal homeomorphism $\varphi : X \to X$ such that $\varphi^* = \beta^*$. It is straightforward to check, using the Pimsner--Voiculescu exact sequence, that $C(X) \rtimes_{\varphi}\Z$ has the required $K$-theory groups.
\end{proof}

Recall that $\mathcal{C}$ denotes the class of simple, separable, unital, infinite-dimensional $\mathrm{C}^*$-algebras with finite nuclear dimension and which satisfy the UCT. In other words, $\mathcal{C}$ consists of those $\mathrm{C}^*$-algebras which are classifiable.

\begin{corollary}
For any $d \in \mathbb{N} \setminus \{0\}$ and any pair of finite abelian groups $F_0, F_1$, there exists a minimal dynamical system $(X, \varphi)$ such for the crossed product $A := C(X) \rtimes_{\varphi} \mathbb{Z}$, we have
\begin{enumerate}
\item $A \in \mathcal{C}$,
\item the pointed ordered $K_0$-group $(K_0(A),  K_0(A)_+,  [1])$ is isomorphic to  $( \mathbb{Z}^d \oplus F_0, \,  \mathbb{Z}_{> 0} \oplus \mathbb{Z}^{d-1} \oplus F_0 \cup (0_{\mathbb{Z}^d},0_{F_0}),\, (1, 0_{\mathbb{Z}^{d-1}}, 0_{F_0}))$,
\item $K_1(A) \cong \mathbb{Z}^d \oplus F_1,$
\item $A$ has a unique tracial state,
\item $r : T(A) \to SK_0(A)$ satisfies $r(\tau)((n_1, \dots, n_d), g)) = n_1$,
\item $A$ has no non-trivial projections. 
\end{enumerate}
\end{corollary}

\begin{proof} By the previous theorem, there exists a minimal dynamical system $(X, \varphi)$ with $K_0(A) \cong  \mathbb{Z}^d \oplus F_0$ and $K_1(A) \cong \mathbb{Z}^d \oplus F_1$. By Theorem~\ref{Thm:MainDynamic} we may moreover choose $X$ such that $H^1(X) = 0$ and $(X, \varphi)$ is uniquely ergodic. As observed in the proof of Proposition~\ref{classifiableTheorem}, unique ergodicity implies that $A \in C$. It also implies that $A = C(X) \rtimes_{\varphi} \mathbb{Z}$ has a unique tracial state. That $A$ has no non-trivial projections follows from the fact that $H^1(X) = 0$, together with a theorem of Connes in \cite{MR605351} (see also \cite[Corollary 10.10.6]{Bla:K-theory}). In this case we can also compute the order structure on $K_0$ and pairing map since the range of the trace on $K_0(A)$ is $\mathbb{Z}$ (see \cite[Corollary 10.10.6]{Bla:K-theory}).
\end{proof}

\section{Technical results on the K-theory of orbit-breaking subalgebras} \label{TechResultSection}
In this section we proved a number of technical results related to the $K$-theory of orbit-breaking subalgebras. Let $\mathrm{C}^*(\mathcal{R}_Y)$ denoted the orbit-breaking subalgebras associated to $C(X)\times_{\varphi}\mathbb{Z}$ and $Y\subseteq X$ a closed, non-empty subset. Recall from Theorem \ref{PutExtSeq} that we have the following exact sequence relating the $K$-theory of $A_Y=\mathrm{C}^*(\mathcal{R}_Y)$ to the $K$-theory of $C(Y)$ and $A=C(X) \times_{\varphi} \mathbb{Z}$:
\begin{displaymath} 
\xymatrix{ K^0(Y) \ar[r] & K_0(A_Y) \ar[r]^-{\iota_*} & K_0(C(X) \rtimes_{\varphi} \mathbb{Z}) \ar[d]^{\partial_{\text{OB}}}\\
K_1(C(X) \rtimes_{\varphi} \mathbb{Z} ) \ar[u]^{\partial_{\text{OB}}} & K_1(A_Y) \ar[l]_-{\iota_*} & K^1(Y). \ar[l]}
\end{displaymath}

Below, we establish some further results on the $K$-theory of the orbit-breaking subalgebras which will be needed. The first lemma and its corollary follow from an inspection of the maps the above six-term exact sequence,  details of which are given in \cite{Put:K-theoryGroupoids}.

\begin{lemma} \label{orbitBreakFuncProp}
Let $(X, \varphi)$ be a minimal dynamical system and let $Y_1$ and $Y_2$ be non-empty closed subsets meeting every $\varphi$-orbit at most once. Suppose that $Y_1 \subset Y_2$ and let $j : Y_1 \to Y_2$ denote the inclusion. Let $A := C(X) \rtimes_{\varphi} \mathbb{Z}$. Then $\iota_i : A_{Y_i} \into A$, $i = 1, 2$ and the two exact sequences of Theorem~\ref{PutExtSeq} are compatible. That is, 
\begin{displaymath} 
\xymatrix{ \cdots \ar[r] & K_1(A) \ar@{=}[d] \ar[r] & K^0(Y_2) \ar[d]_{j^*} \ar[r] & K_0(A_{Y_2}) \ar[d] \ar[r]^-{(\iota_2)_*} & K_0(A) \ar[r] \ar@{=}[d]  & \cdots\\
\cdots \ar[r] & K_1(A) \ar[r] & K^0(Y_1) \ar[r] & K_0(A_{Y_1}) \ar[r]^-{(\iota_1)_*} & K_0(A) \ar[r] & \cdots .}
\end{displaymath}

\end{lemma}

In the next corollary, (\ref{NCP}) was also observed by Phillips in \cite[Theorem 4.1(3)]{MR2320644}.

\begin{corollary} \label{corOrbitBreakYandPoint}
Applying the previous lemma to the case $Y_1=\{y \}$ for some $y \in X$, $Y_2$  connected and $K_1(A) \cong \Z$, we have 
\begin{displaymath} 
\xymatrix{ \cdots \ar[r] &  \mathbb{Z} \ar@{=}[d] \ar[r] &  \mathbb{Z} \oplus \tilde{K}^0(Y_2) \ar[d]_{j^*} \ar[r] & K_0(A_{Y_2}) \ar[d] \ar[r]^-{(\iota_2)_*} & K_0(A) \ar[r] \ar@{=}[d]  & \cdots\\
\cdots \ar[r] & \mathbb{Z} \ar[r] & K^0(\{y\}) \cong \mathbb{Z} \ar[r] & K_0(A_{\{y\}}) \ar[r]^-{(\iota_1)_*} & \ar[r] K_0(A) & \cdots .}
\end{displaymath}

Moreover, 
\begin{enumerate}
\item the map $K_0(A_{\{ y \}}) \rightarrow K_0(A)$ is an isomorphism; \label{NCP}
\item the map $\Z \rightarrow K^0(Y_1)=K^0(\{ y\})\cong \Z$ is an isomorphism; 
\item the map $K^0(Y_2)  \cong \Z \oplus \tilde{K}^0(Y_2) \rightarrow K^0(\{ y \}) \cong \Z$ is given by 
\[ (n, y) \in \Z \oplus \tilde{K}^0(Y_2) \mapsto n \in \Z. \]
\end{enumerate}
\end{corollary}

For the next lemma, recall from Section~\ref{Sect:Prelim} that the boundary maps in the Pimsner--Voiculescu exact sequence are denoted $\partial_{\text{PV}}$, and the boundary maps in the exact sequence of Theorem~\ref{PutExtSeq} are denoted $\partial_{\text{OB}}$.

\begin{lemma} \label{boundaryMapLemma}
Let $(X, \varphi)$ be a minimal dynamical system with $X$ an infinite compact metric space. Suppose $Y \subseteq X$ is a closed subset of $X$ such that $\varphi^{n}(Y) \cap Y = \emptyset$ for every $n \neq 0$. Let $j : Y \into X$ denote the inclusion of $Y$ in $X$. Then the following diagram commutes:
\begin{displaymath} 
\xymatrix{ K_*(A) \ar[d]_{\partial_{\text{OB}}} \ar[r]^-{\partial_{\text{PV\,\,}}} & K^{*+1}(X) \ar[dl]^{j^*} \\
K^{*+1}(Y). }
\end{displaymath}
\end{lemma}

 Before beginning the proof of Lemma \ref{boundaryMapLemma}, we need a more explicit description of the boundary maps. The boundary maps $\partial_{\rm\, PV}$ of the Pimsner--Voiculescu exact sequence are discussed first.
   
Regard the crossed product as being generated by functions in $C(X)$ and a unitary $u$ such that $ufu^{*} = f \circ \varphi^{-1}$, for $f$ in $C(X)$. Let $b \in \mathcal{B}(\ell^{2}(\mathbb{Z}))$ denote the bilateral shift and let $\mathcal{E} \subset (C(X) \times_{\varphi } \mathbb{Z}) \otimes \mathcal{B}(\ell^{2}(\mathbb{Z}))$ be the \mbox{$\mathrm{C}^*$-subalgebra} generated $C(X) \otimes 1$ and $u \otimes b$.  Denote by $p^+ \in \mathcal{B}(\ell^{2}(\mathbb{Z}))$ the projection onto $\ell(\mathbb{Z}_{>0})$ and put $p := 1_{C(X)} \otimes p^+$. There is a short exact sequence
     \[
         0 \rightarrow C(X) \otimes \mathcal{K}(\ell^{2}(\Z_{>0})) 
  \rightarrow p\mathcal{E}p \rightarrow C(X) \times_{\varphi} \Z \rightarrow 0,
  \]       
  and the maps $\partial_{\rm\, PV}$ are the index maps of the associated sequence in $K$-theory.

  To describe the second boundary map $\partial_{\text{OB}}$ we elaborate a little on the $K$-theory of the orbit-breaking 
  subalgebras. Let $A := \mathrm{C}^{*}(\mathcal{R}_{\varphi})$ and let \mbox{$A_{Y} := \mathrm{C}^{*}(\mathcal{R}_{Y})$}.

   Let $\mathbb{Z}$ act on $Y \times \mathbb{Z}$ via the trivial action on $Y$ and translation on $\mathbb{Z}$, and let $\mathcal{S}_{\varphi} $ denote its associated groupoid. We view $\mathcal{S}_\varphi$ as the equivalence relation
   \[
   \mathcal{S}_{\varphi} = 
   \{ \left( (y,m), (y, n) \right) \mid y \in Y, m,n \in \mathbb{Z} \}.
   \]
Define
   \begin{eqnarray*}
   \mathcal{S}_{Y}  & =  &  \mathcal{S} \setminus
   \{ \left( (y,m), (y, n) \right)  \mid y \in Y; m < 1 \leq n 
    \text{ or } n < 1 \leq m 
     \},
   \end{eqnarray*}
and let $C := \mathrm{C}^{*}(\mathcal{S}_{\varphi})$ and $C_{Y} := \mathrm{C}^{*}(\mathcal{S}_{Y})$ be the associated $\mathrm{C}^*$-algebras.  Observe that by defining
    \begin{eqnarray*}
   \mathcal{S}_{Y}^{+}  & :=  &
   \{ \left( (y,m), (y, n) \right) \mid y \in Y, m,n \geq 1\}, \\
   C_{Y}^{+} & = &  C^{*}\left(\mathcal{S}_{Y}^{+}  \right), \\
    \mathcal{S}_{Y}^{-}  & :=  &
   \{ \left( (y,m), (y, n) \right) \mid y \in Y, m,n < 1\}, \\
   C_{Y}^{-} & := &  C^{*}\left(\mathcal{S}_{Y}^{-}  \right),
   \end{eqnarray*}
  we have  $\mathcal{S}_{Y} = \mathcal{S}_{Y}^{+} \cup \mathcal{S}_{Y}^{-}$ and $C_{Y} =    C_{Y}^{+}  \oplus C_{Y}^{-}$. It is straightforward to check (see also \cite[Example 2.6]{Put:K-theoryGroupoids}) that for the associated $\mathrm{C}^*$-algebras we have
   \begin{eqnarray*}
   C & \cong & C(Y) \otimes \mathcal{K}\left(\ell^{2}(\mathbb{Z})\right), \\
   C_{Y}^{+} & \cong & C(Y) \otimes \mathcal{K}\left(\ell^{2}(\mathbb{Z}_{>0})\right), \\
   C_{Y}^{-} & \cong & C(Y) \otimes \mathcal{K}\left(\ell^{2}(\mathbb{Z}_{\leq 0})\right).
   \end{eqnarray*}
Note that the inclusion $C_{Y} \subset C$ is considerably more elementary than that of $A_{Y} \subset A$. We want to establish a relation between these two pairs of inclusions.  The map $j: Y \times \mathbb{Z} \rightarrow X$  defined by
   $j(y,n) = \varphi^{n}(y)$ is continuous and  equivariant. By abuse of notation, we also use $j$ to denote $ j \times j$, which is a  groupoid  morphism from $\mathcal{S}_{\varphi} $ to $\mathcal{R}_{\varphi}$. Observe that $j(\mathcal{S}_{Y} )$ is precisely $j(\mathcal{S}_{\varphi}) \cap \mathcal{R}_{Y}$.
    
    If $g \in C_{c}(\mathcal{R}_{\varphi})$, then $g \circ j: \mathcal{S}_{\varphi} \rightarrow \mathbb{C}$ is continuous and although it does not have compact support, we nevertheless have that for any $f \in C_{c}(\mathcal{S}_{\varphi})$ the convolution products $f  (g \circ j), (g \circ j) f$ are both defined and contained in $C_{c}( \mathcal{S}_{\varphi})$. Thus elements of $A$ act as multipliers on $C$. Similarly, elements of $A_{Y}$ acts as multipliers on $C_{Y}$.  

The main result of \cite{Put:Excision} is that there is an isomorphism between relative $K$-groups, $K_{*}(A_{Y}; A) \cong K_{*}(C_{Y};C)$. Thus we can reduce the computation of $K_{*}(A_{Y}; A)$ to the computation of $K_{*}(C_{Y};C)$. Now, given any $\mathrm{C}^*$-algebra $B$ with $\mathrm{C}^*$-subalgebra $B'$, by \cite[Section 2]{Put:Excision}, the relative $K_0$-group,  $K_0(B', B)$ is shown to consist of  elements  represented by partial isometries in matrices over $\tilde{B}$, the unitization of $B$, whose initial and final projections lie in matrices over $\tilde{B'}$. There is an exact sequence
\begin{displaymath}
\xymatrix{ K_1(B) \ar[r] & K_0(B';B) \ar[r] &K_0(B') \ar[d]  \\
 \ar[u] K_1(B') & \ar[l]  K_1(B';B) & \ar[l] K_0(B) ,
}
\end{displaymath}
given as follows. The map $K_1(B) \to K_0(B';B)$ is defined by considering any unitary in matrices over the unitization $\tilde{B}$ as a partial isometry with initial and final projections in matrices over $\tilde{B'}$, and the map $K_0(B'; B) \to K_0(B')$ is defined by sending a partial isometry $v$ to  $[v^{*}v]_{0} - [vv^{*}]_{0} \in K_{0}(B')$. The vertical maps are induced by the inclusion $B'\subset B$.

With the concrete description of the $\mathrm{C}^*$-algebras $C_Y \subset C$, it is a simple matter to check that this yields a short exact sequence
\[
0 \rightarrow K_{0}(C_{Y}; C) \rightarrow K_{0}(C_{Y}^{+})  \oplus K_{0}(  C_{Y}^{-})  \rightarrow
K_{0}(C)  \rightarrow 0,
\]
and that the inclusions $C_{Y}^{+}, C_{Y}^{-} \subset C$ both induce isomorphisms on $K$-theory. Consequently, if $q:= 1 \oplus 0$ in the multiplier algebra of $C_{Y}^{+} \oplus C_{Y}^{-}$,  the map from $K_{0}(C_{Y};C)$ to $K_{0}(C)$ 
taking the class of a partial isometry $v$ in $C$ with initial and final projections in $C_{Y}^{+} \oplus C_{Y}^{-}$ to $[v^{*}vq]_{0} - [vv^{*}q]_{0}$, is an isomorphism.
    
  We now complete our description of the map $\partial_{\text{OB}}$. For each $m \geq 1$, let $e_{m}$ be the characteristic function of the compact open set 
  \[ \{ (y, i, i) \mid |i| \leq m \} \subset \mathcal{S}_{\alpha}.\]
    These elements form an approximate unit for $C$ which lies in $C_{Y}$. We use the same $e_{m}$ to denote $e_{m} \otimes 1_{n}$ in $M_{n}(C)$.

Let $v \in M_n(A)$ be a unitary, which we regard as a partial isometry with initial and final projections in $M_{n}(A')$. Define
  \[
  v_{m} := \left[ \begin{array}{cc} ve_{m} & 0 \\ 1_{n}-e_{m} & 0 \end{array} \right],
  \]
  which is a partial isometry in $M_{2n}( \tilde{C} )$ satisfying  
  \[v_{m}^{*}v_{m} = 1 \oplus 0, \qquad v_{m}v_{m}^{*} =  ve_{m}v^{*} \oplus (1-e_{m} ).\]
    For sufficiently large values of $m$, $v_{m}v_{m}^{*}$
     will lie (at least approximately) in $M_{2n}(\tilde{C_{Y}})$. The isomorphism 
     from \cite{Put:Excision} carries the class of $v$ to that of $v_{m}$, for any sufficiently large $m$. Thus for the map $\partial_{OB}$ we have 
     \[ [v]_{1} \mapsto [q(1 \oplus 0 )]_{0} - [ q(ve_{m}v^{*} \oplus (1-e_{m} ))]_{0} \in K_{0}(C).\]
    
\noindent\emph{Proof of Lemma \ref{boundaryMapLemma}}. It will be convenient to let $e^{+}_{m}$ denote
  the characteristic function of the set $\{ (y, i, i) \mid 1 \leq i \leq m\}$. In other words, 
  $e^{+}_{m} = pe_{m}$, where, as above, $p = 1 \otimes p^+$ where $p^+$ is the projection onto $\ell(\mathbb{Z}_{<0})$. First, we show that $j^*\circ\partial_{\text{PV}}([u]_1) = \partial_{\text{OB}}([u]_1)$ where $u$ is the canonical unitary in the crossed product. We compute the Pimsner--Voiculescu map as follows. Lift $u$ to
$p(u \otimes b)p = u \otimes s$ in $p\mathcal{E}p$, where $s$ is the unilateral shift.
Then
\[
  \partial_{\text{PV}}([u]_{1}) = [1 \otimes ss^{*}]_{0} - [ 1 \otimes s^{*}s]_{0} = - [1 \otimes (I-ss^{*})]_{0}. 
  \]
  Under the isomorphism of $K_{0}(C(X) \otimes \mathcal{K})$ with $K_{0}(C(X))$, this is identified with
  $- [1]_{0}$. On the other hand,  
  \begin{eqnarray*}
  \partial_{\text{OB}}[u]_{1} & =  & [q(1 \oplus 0)]_{0} - [ q(ue_{m}u^{*} \oplus (1-e_{m} ))]_{0} \\
      &  =  &  
  [q]_{0} -  [ e^{+}_{m+1} \oplus (q-e^{+}_{m} ))]_{0} \\
    & = & -[e_{1}^{+}]_{0}.
  \end{eqnarray*}
  Under the isomorphism between $K_{0}(C)$ and $K_{0}(C(Y))$, 
 $[e_{1}^{+}]_{0}$ is identified with
  $[1]_{0}$. As the restriction map from $C(X)$ to $C(Y)$ is unital, we have $j^*\circ\partial_{\text{PV}}([u]_1)) = \partial_{\text{OB}}([u]_1)$, as required.
    
Now we show that $j^*\circ\partial_{\text{PV}} = \partial_{\text{OB}}$ when applied to an arbitrary unitary in  $M_{n}(\tilde{A})$. For simplicity, assume $n=1$. We may approximate this unitary
  by an invertible $v = \sum_{-N}^{N} f_{n}u^{n}$, where $N \in \mathbb{N}$ and $f_{n} \in C(X)$,  $-N \leq n \leq N$.
  As we know the conclusion holds for $u$, multiplying by $u^N$, it suffices to prove it holds for
  $u^{N}v = \sum_{0}^{2N} f_{n-N}u^{n}$.
  
  For $m > 2N$, the initial and final projections of $(u^{N}v)_{m}$ are in $\tilde{C_{Y}}$,
  and  we have
  \[
  \partial_{\text{OB}}[u^{N}v] = [qu^{N}ve_{m}v^{*}u^{N}]_{0} - [e_{m}^{+}]_{0}.
  \]
  On the other hand, we may lift the element $u^{N}v = \sum_{n=0}^{2N} f_{n-N}u^{n}$ to the element $w = \sum_{n=0}^{2N} f_{n}u^{n} \otimes B^{n}$ in $\mathcal{E}$. From the fact that $p(u \otimes b)p=(u \otimes b)p$, it follows that  $pwp=wp$ and is a partial isometry (approximately) with $(wp)^{*}(wp) \approx p$ and  $(wp)(wp)^{*} = wpw^{*} = pwpw^{*}p$, which is a subprojection of $p$. It follows that 
 \[
 \partial_{\text{PV}}[u^{N}v]_{1} = -[p - wpw^{*}]_{0}.
 \]
  
  We also know that 
   \[
   p(u \otimes b)^{*}p = p(u \otimes b)^{*}
   \]
   and hence, provided $m > 1$, 
   \[
   pw(u \otimes b)^{*m}p = pw(u \otimes b)^{*m}.
   \]
If we let $d_{m} = 1 \otimes p_{\ell^{2}\{-m, \ldots, m \}}$ where $ p_{\ell^{2}\{-m, \ldots, m \}}$ is the projection
   onto \newline $\ell^{2}\{-m, \ldots, m \}$, we have 
  \[
   (1-p)d_{m} =   (1-p)(u \otimes b)^{*(m+1)}p (u \otimes b)^{m+1}.
  \]
It is clear that  $wp(1-d_{m})w^{*}$
  is a subprojection of $wpw^{*}$ and hence also of $p$. It follows that 
  \[
 [p-wpw^{*}]_{0} = [p- wp(1-d_{m})w^{*}]_{0} - [wpw^{*}- wp(1-d_{m})w^{*}]_{0}
  \]
  
  For the second term, we have 
 \[
 [wpw^{*}- wp(1-d_{m})w^{*}]_{0} = [wpd_{m}w^{*}]_{0} = [pd_{m}]_{0}.
  \]
  and for the first, we have
 
   \begin{eqnarray*}
   p - wp(1-d_{m})w^{*} & = & p - wpw^{*} + wpd_{m}w^{*} \\
                  & = &  p - pwpw^{*}p + wpd_{m}w^{*} \\
                   & = &  pw(1-p)w^{*}p + wpd_{m}w^{*} \\
                     &  =  & 
            pw(u \otimes b)^{*(m+1)} (u \otimes b)^{m+1} (1-p)w^{*}p \\
   &  &            +   wpd_{m}w^{*} \\ 
        &  =  &  
          pvw(u \otimes b)^{*(m+1)}p (u \otimes b)^{m+1} (1-p)w^{*}p  \\
   &  &       + wpd_{m}w^{*}  \\
     &  =  &  
          pwd_{m} (1-p)w^{*}p 
     + wpd_{m}w^{*}  \\
     & = & pwe_{m}w^{*}p.
            \end{eqnarray*}

  Applying the restriction map from $C(X)$ to $C(Y)$ takes $p$ to $q$ and 
  $d_{m}$ to $e_{m}$, and we are done.
 \qed \\

\section{Projectionless orbit-breaking algebras} \label{sec:orbitbreaking}

With the technical results of the previous section proved, we can explore the range of the Elliott invariant for orbit-breaking subalgebras. The starting point is results from \cite{DPS:DynZ}, which we summarize in the theorem below.

\begin{theorem} \label{ThmAboutZ}
Let $S^d$  be a sphere with odd dimension $d\geq 3$, and let $\varphi : S^d \to S^d$  be a minimal diffeomorphism. Then there exist an infinite compact metric space $Z$ with finite covering dimension and a minimal homeomorphism \mbox{$\zeta : Z \to Z$} satisfying the following:
\begin{enumerate}
\item $Z$ is compact, connected, and homeomorphic to an inverse limit of contractible metric spaces $(Z_n, d_n)_{n \in \mathbb{N}}$.
\item For any continuous generalized cohomology theory we have an isomorphism $H^*(Z) \cong H^*(\{\mathrm{pt}\})$. In particular this holds for \v{C}ech cohomology and $K$-theory. \label{sameCohomAsPoint}
\item There is an almost one-to-one factor map $q : Z \to S^d$ which induces a bijection between $\zeta$-invariant Borel probability measures on $Z$ and $\varphi$-invariant Borel probability measures on $S^d$. \label{FactorMap}
\end{enumerate} 
\end{theorem}

Let $G_0$ and $G_1$ be arbitrary countable abelian groups. Standard results imply that we can take a compact connected metric space $Y$ with finite dimension and
\[ K^0(Y) \cong \mathbb{Z} \oplus G_0, \qquad K^1(Y)  \cong G_1.\]

We now consider $Y$ fixed for the rest of our discussion in this section. Let $d$ be an odd number large enough such that there exists embedding $Y \hookrightarrow S^{d-2}$. Let $(Z, \zeta)$ be a minimal dynamical system constructed from a minimal diffeomorphism $\varphi : S^d \rightarrow S^d$ as given by Theorem~\ref{ThmAboutZ}.

\begin{lemma}
There exists an embedding $\iota : Y \to Z$ such that $\varphi^n(\iota(Y)) \cap \iota(Y) = \emptyset$ for every $n \in \mathbb{N} \setminus \{0\}$.  
\end{lemma}

\begin{proof}
In the proof Lemma 1.13 in \cite{DPS:DynZ} there is an embedding of $S^{d-2} \rightarrow Z$ whose image lies in the closed set
$\pi^{-1}\{ x \}$, for some  $x$ in $X$. Compose this with the embedding of $Y$ into the sphere $S^{d-2}$ to get an embedding of $Y$ into $Z$. If $n \neq 0$,  we have $\pi( \zeta^{n} (Y)) = \varphi^{n}(x) \neq x = \pi(Y)$, which implies that $\zeta^{n}(Y) \cap Y $ is empty.
\end{proof}

Let $A := C(Z) \times_{\zeta} \mathbb{Z}$ and $A_Y$ denote the orbit-breaking subalgebra of $A$.

\begin{theorem} \label{K(A_Y)}
The $\mathrm{C}^*$-algebra $A_Y$ satisfies the following
\[K_0(A_Y) \cong K^0(Y) \cong \mathbb{Z} \oplus G_{0},\qquad K_1(A_Y) \cong K^1(Y) \cong G_{1},\]
and the positive cone of $K_0(A_Y)$ is given by
\[ K_0(A_Y)_{+} \cong \{ (n, z) \mid n=0, z=0, \text{ or }n > 0 \}.\]
\end{theorem}

\begin{proof}
It follows from \cite{DPS:DynZ} that $K_0(A) \cong \mathbb{Z}$ and $K_1(A) \cong \mathbb{Z}$. We have an exact sequence
\begin{displaymath} 
\xymatrix{ & C(Z) \ar[dl]_{i_1} \ar[dr]^{i_2} &\\
A_Y \ar[rr]_{\iota} & & A}
\end{displaymath}
where  the map $i_2$ induces a map $\mathbb{Z} \cong K_0(C(Z)) \to K_0(A) \cong \mathbb{Z}$. By the Pimsner--Voiculescu exact sequence, as calculated in the proof of \cite[Proposition 2.8]{DPS:DynZ},  the above map is an isomorphism.  Consequently, the map $\iota_* : K_0(A_Y) \to K_0(A)$, as given in the six-term exact sequence of Theorem~\ref{PutExtSeq}, is onto. Thus the sequence of Theorem~\ref{PutExtSeq} becomes 
\begin{displaymath}
\xymatrix{ \mathbb{Z} \oplus G_0  \ar[r] & K_0(A_Y) \ar[r]^{\iota_*} & \mathbb{Z} \ar[d]^0\\
\mathbb{Z}  \ar[u]^L & K_1(A_Y) \ar[l] & G_1. \ar[l]}
\end{displaymath}
Furthermore, the map $K_0(A_Y) \to \mathbb{Z}$ splits. To see this, note that $i_2$ induces an isomorphism on $K_0$. Thus, using the commutativity of the first diagram in the proof, the splitting map is  given by $(i_1)_* \circ (i_2)_*^{-1}$. 

To complete the proof, we need to show that $L : \mathbb{Z} \to \mathbb{Z} \oplus G_0$ is the map $n \mapsto (n, 0)$. This follows from Corollary \ref{corOrbitBreakYandPoint}.
\end{proof}

\begin{corollary} \label{cor:fewProjections} Let $G_0$ and $G_1$ be countable abelian groups, $k \in \mathbb{Z}_{>0}$, and let $\Delta$ be a finite-dimensional Choquet simplex. Then $(\Z \oplus G_0, \Z_{\geq 0} \oplus G_0, [1]=(k,0))$ is an ordered abelian group, and if there is a map 
\[ r : \Delta \to S(\mathbb{Z} \oplus G_0), \qquad \tau \mapsto \left((n,g) \mapsto \frac{n}{k}\right) ,\]
then there exists an amenable minimal equivalence relation $\mathcal{R}$ such that 
\[ \Ell(\mathrm{C}^*(\mathcal{R})) \cong (\Z \oplus G_0, \Z_{\geq 0} \oplus G_0, [1]=(k, 0), G_1, \Delta,  r).\]
In particular if $A \in \mathcal{C}$ and $\Ell(A) =    (\Z \oplus G_0,   \Z_{\geq 0} \oplus G_0, [1]=(k,0), G_1, \Delta,  r )$, then $A$ and  $\mathrm{C}^*(\mathcal{R})$ are $^*$-isomorphic.
\end{corollary}

\begin{proof}
By \cite{DPS:DynZ} there exists a minimal dynamical system $(Z, \zeta)$ with $Z$ a point-like space and simplex of $\zeta$-invariant measures given by $\Delta$. By Theorem~\ref{K(A_Y)}, there is  a non-empty closed subset $Y \subset Z$ meeting every $\zeta$-orbit at most once such that the associated minimal equivalence relation  $\mathcal{R}_Y \subset Z \times Z$ has $K$-theory given by $(K_0(\mathrm{C}^*(\mathcal{R}_Y)), K_0(\mathrm{C}^*(\mathcal{R}_Y)) = ( \Z \oplus G_0, \, \mathbb{Z}_+)$ and $K_1(\mathrm{C}^*(\mathcal{R}_Y)) = G_1$.  We have $[1] = (1,0)$. To arrange that $[1] = (k,0)$, we replace $\mathcal{R}_Y \subset Z \times Z$ with the equivalence relation on $\mathcal{R} \subset (Z \times \{1, \dots, k\}) \times (Z \times \{1, \dots, k\})$ which gives us $\mathrm{C}^*(\mathcal{R}) \cong \mathrm{C}^*(\mathcal{R}_Y)\otimes M_k$.  Since $\iota : \mathrm{C}^*(\mathcal{R}_Y) \into C(Z) \rtimes_{\zeta} \mathbb{Z}$ induces a homoemorphism of tracial state spaces, 
\[ T(\mathrm{C}^*(\mathcal{R})) \cong T(\mathrm{C}^*(\mathcal{R}_Y)) \cong T(C(Z) \rtimes_{\zeta} \mathbb{Z}) \cong \Delta,\]
and since 
\[ r_{\mathrm{C}^*(\mathcal{R}_Y)} :    T(\mathrm{C}^*(\mathcal{R}_Y)) \to S(K_0(\mathrm{C}^*(\mathcal{R}_Y)), \qquad (n, g) \mapsto n \in \mathbb{Z},\]
we have
\[ r_{\mathrm{C}^*(\mathcal{R})} :    T(\mathrm{C}^*(\mathcal{R}_Y)) \to S(K_0(\mathrm{C}^*(\mathcal{R}_Y))), \qquad (n, g) \mapsto \frac{n}{k}.\]
Thus \[ \Ell(\mathrm{C}^*(\mathcal{R})) \cong (\Z \oplus G_0, \Z^+, [1]=k, G_1, \Delta,  r),\] and since $\mathrm{C}^*(\mathcal{R})$ is a simple, separable, nuclear, unital, $\mathcal{Z}$-stable $\mathrm{C}^*$-algebra, Theorem~\ref{ClassThm} implies that for any $A \in \mathcal{C}$ we have $A \cong \mathrm{C}^*(\mathcal{R})$.
\end{proof}
\begin{remark}
In the previous corollary, we observe that by taking the class of the unit to be (1,0) the resulting $\mathrm{C}^*$-algebras contain only the trivial projections 0 and 1.
\end{remark}

\section{Orbit-breaking algebras with real rank zero} \label{Subsect:RR0}

In this section, we consider the crossed products and orbit-breaking subalgebras arising from the construction of minimal homeomorphisms on non-homogeneous spaces. In contrast to the $\mathrm{C}^*$-algebras of the previous subsection, these $\mathrm{C}^*$-algebras will always have a plentiful supply of projections: they all have real rank zero (and this can be read directly from the $K$-theory).

Let $I=[0,1]$. Building on results of Floyd \cite{floyd1949} and Gjerde and Johansen \cite{FloGjeJohSys}, the authors proved the following in \cite{DPS:nonHom}:

\begin{theorem} \label{FloGjeJohSysThm}
Let $(K, \varphi)$ be a minimal system with $K$ the Cantor set and let $n \geq 1$ be a natural number. Then there exists 
 minimal system, 
$(\tilde{K}, \tilde{\varphi})$ with a factor map $\pi: (\tilde{K}, \tilde{\varphi}) 
\rightarrow (K, \varphi)$ such that, for each point $x$ in $K$, $\pi^{-1}\{ x \} \cong I^{n}$ or a single point.
Moreover, both of these possibilities occur and the map $\pi$ induces an isomorphism $\pi^{*}:K^{*}(K) \rightarrow K^{*}(\tilde{K})$.
\end{theorem}

We require a lemma about the pointed ordered $K$-theory of the crossed product $\mathrm{C}^*$-algebra associated to the extension of the Cantor minimal system in the previous theorem.  

\begin{lemma} \label{KthFGJsys}
Using the notation above, let $A := C(K) \rtimes_{\varphi} \mathbb{Z}$ and $B := C(\tilde{K}) \rtimes_{\tilde{\varphi}} \mathbb{Z}$. Then the factor map $\pi : (\tilde{K}, \tilde{\varphi})  \to (K, \varphi)$ induces isomorphisms
\[ (K_0(A), K_0(A)_+, [1_A]) \cong  (K_0(B), K_0(B)_+, [1_B]), \qquad K_1(A) \cong K_1(B),\]
and an affine homeomorphism $T(B) \cong T(A)$.
\end{lemma}
\begin{proof}
The proof for the case that $\pi^{-1}(x)$, $x \in \tilde{K}$, is either a single point or $I$ is given in \cite[Theorem 4]{FloGjeJohSys}. The proof there generalizes with only minor changes to the case when  $\pi^{-1}(x)$ is either a single point or $I^n$ for $n>2$. We note that in \cite{FloGjeJohSys} the notation  $K(K, \varphi)$ is used to denote the $5$-tuple $(K_0(A), K_0(A)_+, [1_A], K_1(A), T(A))$ for $A = C(K) \rtimes_{\varphi} \mathbb{Z}$.
\end{proof}

Let $G_0$ be a simple dimension group, $T$ a countable abelian group and $G_1$ a countable abelian group. There exists a compact finite-dimensional connected metric space $Y$ such that $K^0(Y) \cong \mathbb{Z} \oplus T$ and $K^1(Y)  \cong G_1$. By \cite[Corollary 6.3]{MR1194074}, there exists Cantor minimal system $(K, \psi)$ such that
\[
K_0(C(K) \rtimes_{\psi} \Z) \cong G_0, \qquad K_1(C(K)\rtimes_{\psi} \Z) \cong \Z.
\]
Let $n \in \N$ be large enough so that there exists an embedding $Y \hookrightarrow I^n$ and $(\tilde{K}, \tilde{\psi})$ be the extension of $(K, \psi)$ with $\pi^{-1}(x)$ a point or $\pi^{-1}(x) = I^n$ for each $x\in K$, which is obtained using Theorem \ref{FloGjeJohSysThm}. By Lemma \ref{KthFGJsys}, there is an explicit isomorphism between the $K$-theory of the crossed product $\mathrm{C}^*$-algebras, induced from the factor map $\pi : \tilde{K} \rightarrow K$, which preserves both the order and the class of the unit.
Let $B :=C(\tilde{K}) \rtimes_{\psi}\Z$ and form $B_Y$, the orbit-breaking subalgebra associated to $Y$, where $Y$ is considered as a closed subspace of $\tilde{K}$ via $Y \hookrightarrow I^n \hookrightarrow \tilde{K}$. Since $I^n \cap \psi^k(I^n) = \emptyset$ for $k \neq 0$, we have that $Y\cap \psi^k(Y) = \emptyset$ for $k \neq 0$ and hence  $B_Y$ is simple.

\begin{theorem} \label{thm:KthRR0}
The $K$-theory of the orbit-breaking subalgebra $B_Y$ is 
\[ (K_0(B_Y), K_0(B_Y)_+, [1]) \cong (T \oplus G_0, G_0^+, 1_{G_0}), \qquad K_1(B_Y) \cong G_1. \]
\end{theorem}

\begin{proof}
The proof is similar to the proof of Theorem 4.1 in \cite{Putnam:MinHomCantor}. Using the commutative diagram
\begin{displaymath} 
\xymatrix{ & C(\tilde{K}) \ar[dl]_{i_1} \ar[dr]^{i_2} &\\
B_Y \ar[rr]_{\iota} & & B,}
\end{displaymath}
and Theorem~\ref{PutExtSeq}, we obtain
\begin{displaymath}
\xymatrix{ \mathbb{Z} \oplus T  \ar[r] & K_0(B_Y) \ar[r]^{\iota_*} & G_0 \ar[d]^0\\
\mathbb{Z}  \ar[u]^L & K_1(B_Y) \ar[l] & G_1. \ar[l]}
\end{displaymath}
Using Corollary \ref{corOrbitBreakYandPoint}, it follows that
\[
0 \rightarrow \mathbb{Z} \rightarrow \Z \oplus T  \rightarrow K_0(B_Y) \rightarrow G_0 \rightarrow 0, 
\]
and $K_1(B_Y)\cong G_1 $. Moreover, the map $\Z \rightarrow \Z \oplus T$ is given by $l \mapsto (l, 0)$, so we have the short exact sequence
\[
0 \rightarrow T \rightarrow K_0(B_Y) \rightarrow G_0 \rightarrow 0.
\]
To complete the first part of the proof, we show that this sequence splits. To do so, consider the maps $Y \rightarrow I^n$ and associated orbit-breaking subalgebras. Lemma \ref{orbitBreakFuncProp} implies that we have 
\begin{displaymath}
\xymatrix{ \ar[r] & \Z \ar[r] \ar[d] & K_0(B_{I^n}) \ar[r] \ar[d] & K_0(C(\tilde{K})\rtimes \Z) \ar[r] \ar@{=}[d] &   \\
\ar[r] & \Z \oplus T \ar[r] &  K_0(B_Y)\ar[r]& K_0(C(\tilde{K})\rtimes \Z) \ar[r] &. }
\end{displaymath}
Moreover, the map $K_0(B_{I^n}) \rightarrow K_0(C(\tilde{K})\rtimes \Z) \cong G_0$ is an isomorphism, which gives us the required splitting.

Next we show that the positive cone of $K_0(B_Y)$ is $(G_0)_+$. We follow the proof of Theorem 4.1 in \cite{Putnam:MinHomCantor}. The result will follow by showing that given $a\in (G_0)_+$ there exists $b\in K_0(B_Y)_+$ such that $\iota_*(b)=a$. As such, let $a\in (G_0)_+$. Observe that for the diagram
\begin{displaymath} 
\xymatrix{ & C(\tilde{K}) \ar[dl]_{i_1} \ar[dr]^{i_2} &\\
B_Y \ar[rr]_{\iota} & & B}
\end{displaymath}
there exists $c\in K^0(\tilde{K})_+$ such that $(i_2)_*(c)=a$. Then the element $b=(i_1)_*(c)$ has the required property. Finally, the class of the unit is respected by the map $\iota_*$ because $\iota$ is unital.
\end{proof}

\begin{corollary}
Let $A$ be a simple, separable, unital $\mathrm{C}^*$-algebra with finite decomposition rank, real rank zero and which satisfies the UCT. Suppose that 
\[ K_0(A) \cong T \oplus G_0, \qquad K_1(A) \cong G_1, \] 
where $T \subset \mathrm{Inf}(K_0(A))$ is a countable abelian group, $G_0$ is a simple dimension group, $G_1$ is a countable abelian group and the order structure and class of the unit of $K_0(A)$ are the same as the simple dimension group $G_0$.
Then there exists an amenable equivalence relation, $\mathcal{R}$, with 
$\mathrm{C}^{*}(\mathcal{R}) \cong A$.
\end{corollary}

\begin{proof}
Since $A$ has real rank zero, the tracial state space and pairing map in the Elliott invariant of $A$ are redundant. Thus the isomorphism class of $A$ consists of $\mathrm{C}^*$-algebras with real rank zero and isomorphic $K$-theory.  
\end{proof}

\subsection*{Acknowledgements}  The authors thank the Banff International Research Stations and the organizers of the workshop Future Targets in the Classification Program for Amenable $\mathrm{C}^*$-Algebras, where this project was initiated. Thanks also to the University of Victoria and University of Colorado, Boulder for research visits facilitating this collaboration. Work on the project also benefitted from the 2018 conference on Cuntz--Pimsner algebras at the Lorentz Center, which was attended by  the first and third listed authors. 


\begin{thebibliography}{10}

\bibitem{MR1799683}
C.~Anantharaman-Delaroche and J.~Renault.
\newblock {\em Amenable groupoids}, volume~36 of {\em Monographies de
  L'Enseignement Math\'{e}matique [Monographs of L'Enseignement
  Math\'{e}matique]}.
\newblock L'Enseignement Math\'{e}matique, Geneva, 2000.
\newblock With a foreword by Georges Skandalis and Appendix B by E. Germain.

\bibitem{ArBkPh-Z}
D.~Archey, J.~Buck, and N.~C. Phillips.
\newblock Centrally large subalgebras and tracial {$\mathcal Z$} absorption.
\newblock {\em Int. Math. Res. Not. IMRN}, (6):1857--1877, 2018.

\bibitem{ArcPhi2016}
D.~Archey and N.~C. Phillips.
\newblock Permanence of stable rank one for centrally large subalgebras and
  crossed products by minimal homeomorphisms.
\newblock {To appear in J. Operator Theory}.

\bibitem{AusMitra:GrpModGelfDual}
K.~Austin and A.~Mitra.
\newblock Groupoid models of {$\mathrm{C}^*$}-algebras and {G}elfand duality.
\newblock {Preprint arXiv:1804.00967}, 2018.

\bibitem{BarLi:CartanUCT1}
S.~Barlak and X.~Li.
\newblock Cartan subalgebras and the {UCT} problem.
\newblock {\em Adv. Math.}, 316:748--769, 2017.

\bibitem{BarLi:CartanUCT2}
S.~Barlak and X.~Li.
\newblock Cartan subalgebras and the {UCT} problem, {II}.
\newblock {Preprint arXiv:1704.04939}, 2017.

\bibitem{Bla:K-theory}
B.~Blackadar.
\newblock {\em {$K$}-theory for operator algebras}, volume~5 of {\em
  Mathematical Sciences Research Institute Publications}.
\newblock Cambridge University Press, Cambridge, second edition, 1998.

\bibitem{BBSTWW:2Col}
J.~Bosa, N.~P. Brown, Y.~Sato, A.~Tikuisis, S.~White, and W.~Winter.
\newblock {Covering dimension of $\mathrm{C}^*$-algebras and 2-coloured classification}.
\newblock {\em Mem. Amer. Math. Soc.}, 257(1233):vii+97, 2019.

\bibitem{CETWW}
J.~Castillejos, S.~Evington, A.~Tikuisis, S.~White, and W.~Winter.
\newblock Nuclear dimension of simple {$\mathrm{C}^*$}-algebras.
\newblock {Preprint arXiv:1901.05853v2}, 2019.

\bibitem{MR605351}
A.~Connes.
\newblock An analogue of the {T}hom isomorphism for crossed products of a
\mbox{{$\mathrm{C}^*$}-algebra} by an action of {${\bf R}$}.
\newblock {\em Adv. in Math.}, 39(1):31--55, 1981.

\bibitem{MR662736}
A.~Connes, J.~Feldman, and B.~Weiss.
\newblock An amenable equivalence relation is generated by a single
  transformation.
\newblock {\em Ergodic Theory Dynamical Systems}, 1(4):431--450 (1982), 1981.

\bibitem{DPS:DynZ}
R.~J. Deeley, I.~F. Putnam, and K.~R. Strung.
\newblock {Constructing minimal homeomorphisms on point-like spaces and a
  dynamical presentation of the Jiang--Su algebra}.
\newblock {\em {J. Reine Angew. Math.}}, {742}:{241--261}, {2018}.

\bibitem{DPS:nonHom}
R.~J. Deeley, I.~F. Putnam, and K.~R. Strung.
\newblock Non-homogeneous extensions of {C}antor minimal systems.
\newblock {To appear in Proc. Amer. Math. Soc.}, 2019.

\bibitem{DPS:MinHomKtheory}
R.~J. Deeley, I.~F. Putnam, and K.~R. Strung.
\newblock Minimal homeomorphisms and topological {$K$}-theory.
\newblock {P}reprint, 2020.

\bibitem{EffrosHahn1967}
E.~G. Effros and F.~Hahn.
\newblock {\em Locally compact transformation groups and {$\mathrm{C}^*$}-
  algebras}.
\newblock Memoirs of the American Mathematical Society, No. 75. American
  Mathematical Society, Providence, R.I., 1967.

\bibitem{Ell:AF}
G.~A. Elliott.
\newblock {On the classification of inductive limits of sequences of semisimple
  finite-dimensional algebras}.
\newblock {\em J. Algebra}, 38:29--44, 1976.

\bibitem{EllGonLinNiu:ClaFinDecRan}
G.~A. Elliott, G.~Gong, H.~Lin, and Z.~Niu.
\newblock On the classification of simple amenable {$\mathrm{C}^*$}-algebras
  with finite decomposition rank {II}.
\newblock {Preprint arXiv:1507.03437v2}, 2015.

\bibitem{EllNiu:MeanDimZero}
G.~A. Elliott and Z.~Niu.
\newblock The {$\mathrm{C}^*$}-algebra of a minimal homeomorphism of zero mean
  dimension.
\newblock {\em Duke Math. J.}, 166(18):3569--3594, 2017.

\bibitem{FatHer:Diffeo}
A.~Fathi and M.~R. Herman.
\newblock Existence de diff\'{e}omorphismes minimaux.
\newblock pages 37--59. Ast\'{e}risque, No. 49, 1977.

\bibitem{floyd1949}
E.~E. Floyd.
\newblock A nonhomogeneous minimal set.
\newblock {\em Bull. Amer. Math. Soc.}, 55(10):957--960, 10 1949.

\bibitem{GioKerr:Subshifts}
J.~Giol and D.~Kerr.
\newblock {Subshifts and perforation}.
\newblock {\em J. Reine Angew. Math.}, 639:107--119, 2010.

\bibitem{MR2054051}
T.~Giordano, I.~Putnam, and C.~Skau.
\newblock Affable equivalence relations and orbit structure of {C}antor
  dynamical systems.
\newblock {\em Ergodic Theory Dynam. Systems}, 24(2):441--475, 2004.

\bibitem{GioPutSkau:orbit}
T.~Giordano, I.~F. Putnam, and C.~F. Skau.
\newblock {Topological orbit equivalence and \mbox{$\mathrm{C}^*$-crossed} products}.
\newblock {\em J. Reine Angew. Math.}, 469:51--111, 1995.

\bibitem{FloGjeJohSys}
R.~Gjerde and {\O}.~Johansen.
\newblock {$\mathrm{C}^*$}-algebras associated to non-homogeneous minimal
  systems and their {K}-theory.
\newblock {\em Math. Scand.}, 85(1):87--104, 1999.

\bibitem{GlaWei:MinSkePro}
S.~Glasner and B.~Weiss.
\newblock {On the construction of minimal skew products}.
\newblock {\em Israel J. Math.}, 34(4):321--336, 1980.

\bibitem{GongJiangSu:Z}
G.~Gong, X.~Jiang, and H.~Su.
\newblock Obstructions to {$\mathcal{Z}$}-stability for unital simple
  \mbox{$\mathrm{C}^*$-algebras}.
\newblock {\em Canadian Math. Bull.}, 43(4):418--426, 2000.

\bibitem{GongLinNiue:ZClass}
G.~Gong, H.~Lin, and Z.~Niu.
\newblock Classification of finite simple amenable {$\mathcal{Z}$}-stable
  \mbox{{$\mathrm{C}^*$}-algebras}.
\newblock {Preprint arXiv:1501.00135v6}, 2015.

\bibitem{MR1742309}
M.~Gromov.
\newblock Topological invariants of dynamical systems and spaces of holomorphic
  maps. {I}.
\newblock {\em Math. Phys. Anal. Geom.}, 2(4):323--415, 1999.

\bibitem{MR1194074}
R.~H. Herman, I.~F. Putnam, and C.~F. Skau.
\newblock Ordered {B}ratteli diagrams, dimension groups and topological
  dynamics.
\newblock {\em Internat. J. Math.}, 3(6):827--864, 1992.

\bibitem{HirWinZac:RokDim}
I.~Hirshberg, W.~Winter, and J.~Zacharias.
\newblock {Rokhlin dimension and $\mathrm{C}^*$-dynamics}.
\newblock {\em Comm. Math. Phys.}, 335(2):637--670, 2015.

\bibitem{Izu:FreeRok1}
M.~Izumi.
\newblock {Finite group actions on {$\mathrm{C}^*$}-algebras with the {R}ohlin property.
  {I}},.
\newblock {\em Duke Math. J.}, 122(2):233--280, 2004.

\bibitem{JiaSu:Z}
X.~Jiang and H.~Su.
\newblock On a simple unital projectionless {$\mathrm{C}^*$}-algebra.
\newblock {\em Amer. J. Math.}, 121(2):359--413, 1999.

\bibitem{KirWinter:dr}
E.~Kirchberg and W.~Winter.
\newblock Covering dimension and quasidiagonality.
\newblock {\em Internat. J. Math.}, 15(1):63--85, 2004.

\bibitem{Kis:RohlinUHF}
A.~Kishimoto.
\newblock The {R}ohlin property for shifts on {UHF} algebras and automorphisms
  of {C}untz algebras.
\newblock {\em J. Funct. Anal.}, 140(1):100--123, 1996.

\bibitem{Kri:DimFun}
W.~Krieger.
\newblock On dimension functions and topological {M}arkov chains.
\newblock {\em Invent. Math.}, {56}({3}):{239--250}, 1980.

\bibitem{Li:Cartan}
X.~Li.
\newblock Constructing {C}artan subalgebras in classifiable stably finite
  {$\mathrm{C}^*$}-algebras.
\newblock {Preprint arXiv:1802.01190v2}, 2018.

\bibitem{Lin:MinDyn}
H.~Lin.
\newblock Crossed products and minimal dynamical systems.
\newblock {\em J. Topol. Anal.}, 10(2):447--469, 2018.

\bibitem{LinWeiss:MTD}
E.~Lindenstrauss and B.~Weiss.
\newblock Mean topological dimension.
\newblock {\em Israel J. Math.}, 115:1--24, 2000.

\bibitem{MR3273581}
H.~Matui and Y.~Sato.
\newblock Decomposition rank of {UHF}-absorbing {$\mathrm{C}^*$}-algebras.
\newblock {\em Duke Math. J.}, 163(14):2687--2708, 2014.

\bibitem{MvN:Rings4}
F.~J. Murray and J.~von Neumann.
\newblock On rings of operators. {IV}.
\newblock {\em Ann. of Math. (2)}, {44}:{716--808}, 1943.

\bibitem{MR2320644}
N.~C. Phillips.
\newblock Cancellation and stable rank for direct limits of recursive
  subhomogeneous algebras.
\newblock {\em Trans. Amer. Math. Soc.}, 359(10):4625--4652, 2007.

\bibitem{Phi:LargeSubalgebras}
N.~C. Phillips.
\newblock Large subalgebras.
\newblock {Preprint arXiv:1408.5546}, 2014.

\bibitem{Pimsner1980}
M.~Pimsner and D.~Voiculescu.
\newblock Exact sequences for {$K$}-groups and {E}xt-groups of certain
  cross-product {$\mathrm{C}^*$}-algebras.
\newblock {\em J. Operator Theory}, 4(1):93--118, 1980.

\bibitem{Putnam:MinHomCantor}
I.~F. Putnam.
\newblock {The $\mathrm{C}^*$-algebras associated with minimal homeomorphisms of the
  Cantor set}.
\newblock {\em Pacific J. Math.}, 136(2):329--353, 1989.

\bibitem{Put:Excision}
I.~F. Putnam.
\newblock {An excision theorem for the {$K$}-theory of {$\mathrm{C}^*$}-algebras}.
\newblock {\em {J. Operator Theory}}, {38}({1}):{151--171}, {1997}.

\bibitem{Put:K-theoryGroupoids}
I.~F. Putnam.
\newblock On the {$K$}-theory of {$\mathrm{C}^*$}-algebras of principal groupoids.
\newblock {\em {Rocky Mountain J. Math.}}, {28}({4}):{1483--1518}, {1998}.

\bibitem{MR3770169}
I.~F. Putnam.
\newblock Some classifiable groupoid {$\mathrm{C}^*$}-algebras with prescribed
  {$K$}-theory.
\newblock {\em Math. Ann.}, 370(3-4):1361--1387, 2018.

\bibitem{MR584266}
J.~Renault.
\newblock {\em A groupoid approach to {$\mathrm{C}^*$}-algebras}, volume 793 of
  {\em Lecture Notes in Mathematics}.
\newblock Springer, Berlin, 1980.

\bibitem{RosSho:UCT}
J.~Rosenberg and C.~Schochet.
\newblock {The K\"unneth Theorem and the Universal Coefficient Theorem for
  Kasparov's generalized $K$-functor}.
\newblock {\em Duke Math. J.}, 55(2):431--474, 1987.

\bibitem{MR3418247}
Y.~Sato, S.~White, and W.~Winter.
\newblock Nuclear dimension and {$\mathcal{Z}$}-stability.
\newblock {\em Invent. Math.}, 202(2):893--921, 2015.

\bibitem{Takesaki1967}
M.~Takesaki.
\newblock Covariant representations of {$\mathrm{C}^*$}-algebras and their locally
  compact automorphism groups.
\newblock {\em Acta Math.}, 119:273--303, 1967.

\bibitem{TWW}
A.~Tikuisis, S.~White, and W.~Wilhelm.
\newblock {Quasidiagonality of nuclear $\mathrm{C}^*$-algebras}.
\newblock {\em {Ann. of Math.}}, 185(1):229--284, 2017.

\bibitem{TomsWinter:PNAS}
A.~S. Toms and W.~Winter.
\newblock Minimal dynamics and the classification of {$\mathrm{C}^*$}-algebras.
\newblock {\em Proc. Natl. Acad. Sci. USA}, 106(40):16942--16943, 2009.

\bibitem{TomsWinter:minhom}
A.~S. Toms and W.~Winter.
\newblock Minimal {D}ynamics and {K}-{T}heoretic {R}igidity: {E}lliott's
  {C}onjecture.
\newblock {\em Geom. Funct. Anal.}, 23(1):467--481, 2013.

\bibitem{Tu:Groupoids}
J.-L. Tu.
\newblock {La conjecture de {B}aum-{C}onnes pour les feuilletages moyennables}.
\newblock {\em {$K$-Theory}}, {17}({3}):{215--264}, {1999}.

\bibitem{Win:cpr}
W.~Winter.
\newblock Covering dimension for nuclear {$\mathrm{C}^*$}-algebras.
\newblock {\em J. Funct. Anal}, 199(2):535--556, 2003.

\bibitem{Win:Z-stabNucDim}
W.~Winter.
\newblock {Nuclear dimension and {$\mathcal{Z}$}-stability of pure
  {$\mathrm{C}^*$}-algebras}.
\newblock {\em Invent. Math.}, 187(2):259--342, 2012.

\bibitem{WinterZac:dimnuc}
W.~Winter and J.~Zacharias.
\newblock The nuclear dimension of {$\mathrm{C}^*$}-algebras.
\newblock {\em Adv. Math.}, 224(2):461--498, 2010.

\bibitem{Zeller-Meier1968}
G.~Zeller-Meier.
\newblock Produits crois\'{e}s d'une {$\mathrm{C}^*$}-alg\`ebre par un groupe
  d'automorphismes.
\newblock {\em J. Math. Pures Appl. (9)}, 47:101--239, 1968.

\end{thebibliography}

\end{document}